\documentclass[11pt,reqno,tbtags,a4paper]{amsart}
\usepackage{amssymb}
\usepackage{mathabx}  % \widecheck
\usepackage{xpunctuate}
\usepackage{url}
\usepackage[square,numbers]{natbib}
\bibpunct[, ]{[}{]}{;}{n}{,}{,}

\title%[]
{Uncovering a graph}

\date{21 December, 2023 
(typeset \today{} \klockan)}   %\Small
\author{Svante Janson}
\thanks{Supported by the Knut and Alice Wallenberg Foundation}
\address{Department of Mathematics, Uppsala University, PO Box 480,
SE-751~06 Uppsala, Sweden}
\email{svante.janson@math.uu.se}
\newcommand\urladdrx[1]{{\urladdr{\def~{{\tiny$\sim$}}#1}}}
\urladdrx{http://www.math.uu.se/~svante/}
  %05C05 trees? 
\numberwithin{equation}{section}

\renewcommand\le{\leqslant}
\renewcommand\ge{\geqslant}

\allowdisplaybreaks

\theoremstyle{plain}% default
\newtheorem{theorem}{Theorem}[section]
\newtheorem{lemma}[theorem]{Lemma}
\newtheorem{proposition}[theorem]{Proposition}
\newtheorem{corollary}[theorem]{Corollary}

\theoremstyle{definition}

\newcommand\xqed[1]{%
    \leavevmode\unskip\penalty9999 \hbox{}\nobreak\hfill
    \quad\hbox{#1}}

\newtheorem{exampleqqq}[theorem]{Example}
\newenvironment{example}{\begin{exampleqqq}}
  {\xqed{$\triangle$}\end{exampleqqq}}
\newtheorem{remarkqqq}[theorem]{Remark}
\newenvironment{remark}{\begin{remarkqqq}}
  {\xqed{$\triangle$}\end{remarkqqq}}
\newtheorem{problem}[theorem]{Problem}

\newtheorem*{ack}{Acknowledgement}

\theoremstyle{remark}

\newenvironment{romenumerate}[1][-10pt]{% optional argument changes indentation
\addtolength{\leftmargini}{#1}\begin{enumerate}% gives (i), (ii) etc.
 \renewcommand{\labelenumi}{\textup{(\roman{enumi})}}%
 }{\end{enumerate}}

\newcounter{oldenumi}
{\setcounter{oldenumi}{\value{enumi}}
\begin{romenumerate} \setcounter{enumi}{\value{oldenumi}}}
{\end{romenumerate}}

\newcounter{thmenumerate}

\newcounter{xenumerate}   %no left indentation; thus wider lines

\newcommand\xfootnote[1]{\unskip\footnote{#1}$ $} %$ $ tycks eliminera fel
\newcommand\pfitemx[1]{\par#1:}
\newcommand\pfitemref[1]{\pfitemx{\ref{#1}}}

\newcommand{\refT}[1]{Theorem~\ref{#1}}
\newcommand{\refTs}[1]{Theorems~\ref{#1}}
\newcommand{\refC}[1]{Corollary~\ref{#1}}

\newcommand{\refL}[1]{Lemma~\ref{#1}}
\newcommand{\refLs}[1]{Lemmas~\ref{#1}}
\newcommand{\refR}[1]{Remark~\ref{#1}}

\newcommand{\refS}[1]{Section~\ref{#1}}
\newcommand{\refSs}[1]{Sections~\ref{#1}}
\newcommand{\refSS}[1]{Section~\ref{#1}}
\newcommand{\refP}[1]{Proposition~\ref{#1}}

\newcommand{\refE}[1]{Example~\ref{#1}}

\newcommand{\refApp}[1]{Appendix~\ref{#1}}
\newcommand\marginal[1]{\marginpar[\raggedleft\tiny #1]{\raggedright\tiny#1}}
\newcommand\SJ{\marginal{SJ} }

\newcommand\REM[1]{{\raggedright\texttt{[#1]}\par\marginal{XXX}}}
\newcommand\XREM[1]{\relax}

\begingroup
  \divide\count255 by 60
  \multiply\count255 by -60
  \advance\count255 by \time
  \ifnum \count255 < 10 \xdef\klockan{\the\count1.0\the\count255}
  \else\xdef\klockan{\the\count1.\the\count255}\fi
\endgroup

\DeclareMathOperator*{\sumx}{\sum\nolimits^{*}}

\newcommand{\sumin}{\sum_{i=1}^n}

\newcommand\set[1]{\ensuremath{\{#1\}}}
\newcommand\bigset[1]{\ensuremath{\bigl\{#1\bigr\}}}

\newcommand\xpar[1]{(#1)}
\newcommand\bigpar[1]{\bigl(#1\bigr)}
\newcommand\Bigpar[1]{\Bigl(#1\Bigr)}
\newcommand\biggpar[1]{\biggl(#1\biggr)}
\newcommand\lrpar[1]{\left(#1\right)}
\newcommand\bigsqpar[1]{\bigl[#1\bigr]}
\newcommand\sqpar[1]{[#1]}

\newcommand\cpar[1]{\{#1\}}

\newcommand\abs[1]{\lvert#1\rvert}
\newcommand\bigabs[1]{\bigl\lvert#1\bigr\rvert}

\def\rompar(#1){\textup(#1\textup)}    % usage: \rompar(...)

\def\xexp(#1){e^{#1}}
\newcommand\ceil[1]{\lceil#1\rceil}

\newcommand\floor[1]{\lfloor#1\rfloor}

\newcommand\setn{\set{1,\dots,n}}

\newcommand\ntoo{\ensuremath{{n\to\infty}}}
\newcommand\Ntoo{\ensuremath{{N\to\infty}}}
\newcommand\Mtoo{\ensuremath{{M\to\infty}}}

\newcommand\bmin{\land}

\newcommand\upto{\nearrow}
\newcommand\punkt{\xperiod}    % xpunctuate
\newcommand\iid{i.i.d\punkt}    
\newcommand\ie{i.e\punkt}
\newcommand\eg{e.g\punkt}

\newcommand\cf{cf\punkt}
\newcommand{\as}{a.s\punkt}

\newcommand{\tend}{\longrightarrow}
\newcommand\dto{\overset{\mathrm{d}}{\tend}}
\newcommand\pto{\overset{\mathrm{p}}{\tend}}

\newcommand\eqd{\overset{\mathrm{d}}{=}}

\newcommand\op{o_{\mathrm p}}
\newcommand\Op{O_{\mathrm p}}
\newcommand\opx{\op^*}
\newcommand\Opx{\Op^*}

\newcommand\bbR{\mathbb R}

\newcommand\bbN{\mathbb N}

\newcounter{CC}
\newcounter{cc}
\newcommand\E{\operatorname{\mathbb E}{}}
\renewcommand\P{\operatorname{\mathbb P{}}}

\newcommand\Var{\operatorname{Var}{}}
\newcommand\Cov{\operatorname{Cov}}

\newcommand\Po{\operatorname{Po}}

\newcommand\Bin{\operatorname{Bin}}

\newcommand\Ge{\operatorname{Ge}}

\newcommand\ga{\alpha}
\newcommand\gb{\beta}
\newcommand\gd{\delta}
\newcommand\gD{\Delta}
\newcommand\gf{\varphi}

\newcommand\gl{\lambda}

\newcommand\gs{\sigma}
\newcommand\gS{\Sigma}

\newcommand\gU{\Upsilon}
\newcommand\eps{\varepsilon}
\renewcommand\phi{\xxx}  %% WARNING

\newcommand\cF{\mathcal F}
\newcommand\cG{\mathcal G}

\newcommand\cL{{\mathcal L}}

\newcommand\cP{\mathcal P}

\newcommand\cT{{\mathcal T}}

\newcommand\tM{\widetilde M}

\newcommand\tY{{\tilde Y}}

\newcommand\indic[1]{\boldsymbol1\cpar{#1}}

\newcommand\qw{^{-1}}
\newcommand\qww{^{-2}}
\newcommand\qq{^{1/2}}
\newcommand\qqw{^{-1/2}}

\newcommand\intot{\int_0^t}

\newcommand\oi{\ensuremath{[0,1]}}
\newcommand\oio{\ensuremath{[0,1)}}

\newcommand\ooo{[0,\infty)}

\newcommand\dd{\,\mathrm{d}}

\newcommand{\gsf}{$\gs$-field}

\newcommand\lhs{left-hand side}
\newcommand\rhs{right-hand side}

\newcommand\GW{Galton--Watson}
\newcommand\GWt{\GW{} tree}
\newcommand\cGWt{conditioned \GW{} tree}

\newcommand\Uoi{\mathsf{U}(0,1)}
\newcommand\xoo{_1^\infty}

\newcommand\nn{^{(n)}}

\newcommand\discrete[1]{\dot{#1}}
\newcommand\iB{\discrete{B}}
\newcommand\iK{\discrete{K}}
\newcommand\iL{\discrete{L}}
\newcommand\iT{\discrete{T}}
\newcommand\iW{\discrete{W}}
\newcommand\iX{\discrete{X}}
\newcommand\iY{\discrete{Y}}
\newcommand\iZ{\discrete{Z}}
\newcommand\igs{\discrete{\gs}}

\newcommand\bcZ{Z}
\newcommand\II{\widetilde I}
\newcommand\III{\widecheck{I}}
\newcommand\NN{\widetilde N}
\newcommand\QQ{\widetilde Q}
\newcommand\RR{\widetilde R}
\renewcommand\SS{\widetilde S}
\newcommand\TT{\widetilde T}
\newcommand\ZZ{\widetilde Z}
\newcommand\gSS{\widetilde \gS}
\newcommand\NNN{\widecheck N}
\newcommand\QQQ{\widecheck Q}
\newcommand\RRR{\widecheck R}
\newcommand\SSS{\widecheck S}
\newcommand\TTT{\widecheck T}
\newcommand\XXXX{\widecheck X}
\newcommand\YYY{\widecheck Y}
\newcommand\ZZZ{\widecheck Z}
\newcommand\gSSS{\widecheck \Sigma}
\newcommand\ctime{con\-tin\-u\-ous-time} %con-tin-u-ous
\newcommand\dtime{dis\-crete-time}
\newcommand\BR{W^\circ}
\newcommand\fnt{{\floor{nt}}}
\newcommand\sK{\mathsf{K}}
\newcommand\sC{\mathsf{C}}
\newcommand\sP{\mathsf{P}}
\newcommand\sQ{\mathsf{Q}}
\newcommand\sR{\mathsf{R}}
\newcommand\sS{\mathsf{S}}
\newcommand\sN{\mathsf{N}}
\newcommand\tr{'}
\newcommand\MMM{M_1}
\newcommand\gamm{\gamma}
\newcommand\gamx{\gamma_*}
\newcommand\hd{\hat d}
\newcommand\hD{\widehat D}
\newcommand\hgD{\widehat \gD}
\newcommand\doo{d_*} % for regular case
\newcommand\dx{d_*}
\newcommand\chix{\chi_*}
\newcommand\hchi{\widehat\chi}
\newcommand\hchix{\widehat\chi_*}
\newcommand\bd{\bar d}

\newcommand\maxdx{\Delta}
\newcommand\pmaxdx{\maxdx}
\newcommand\nnq{\relax}
\newcommand\bt{\mathbf t}
\newcommand{\Polya}{P\'olya}
\newcommand\CS{Cauchy--Schwarz}
\newcommand\CSineq{\CS{} inequality}

\newcommand\ER{Erd\H os--R\'enyi}

\begin{document}

\begin{abstract} 
Uncover the vertices of a given graph, deterministic or random, in random order;
we consider both a discrete-time and a continuous-time version.
We study the evolution of the number of visible edges, and show convergence
after normalization to a Gaussian process. This problem was studied by
Hackl, Panholzer, and Wagner for the case when the graph is a random
labelled tree; we generalize their result to more general graphs, including
both other classes of random and non-random trees, and denser graphs.
The results are similar in all cases, but some differences can be seen
depending on the size of the average degree and of the variance of the
vertex degrees.
\end{abstract}

\maketitle

\section{Introduction}\label{S:intro}

Let $G$ be a (finite)
graph, deterministic or random, and uncover its vertices one by
one, in uniformly random order;
we say that a vertex becomes \emph{visible} when it is uncovered.
This yields a growing sequence of (random) induced subgraphs of $G$, and we are
interested in the evolution of this sequence.
In particular, we study in this paper
the evolution of the number of edges in these subgraphs,
regarded as a stochastic process.
More precisely, we consider 
a sequence of graphs $G_n$ with order $|G_n|=n$, 
and study the asymptotic behaviour of this stochastic process
as \ntoo{},
under suitable conditions.
(See \refS{Snot} for more details, and for definitions of notation used below.)
The methods extend to the number of other small subgraphs, see \refS{Ssmall}.

This question (among others)
was studied by \citet{HPW} for the case when $G$ is a 
random labelled tree. They showed that the stochastic process %$(\iL_k)_0^n$ 
given by the number of visible edges, 
after suitable rescaling, 
converges to a continuous Gaussian process, which resembles a Brownian
bridge but with a somewhat different distribution; see \refE{Elabelledtree}.
Our main result 
is that this extends to a wide class of deterministic and
random trees and graphs, see \refS{Smain}.

\begin{remark}
  If $G$ is a random graph with vertex set $[n]=\setn$, we may alternatively
  consider uncovering the vertices in the given order.
(Actually, this is the formulation used in \cite{HPW}.)
For a random graph with a distribution that is invariant under permutations
of the vertices (and in particular for the random tree in \cite{HPW}), 
this is obviously equivalent to taking the vertices in
random order, and we will for convenience use only the formulation above.
\end{remark}

We will consider two versions of the problem. In the first, the vertices are
uncovered at fixed times (as in \cite{HPW}); 
in the second, they are uncovered at random times
which are independent for different vertices. The two versions are related
by a simple random change of time. 
We find it interesting to give results for both versions, 
and see their similarities and
differences under different conditions on the graphs $G_n$.

The second version means that at every given time $t$, each vertex is uncovered
with some probability $p=p(t)$ independently of all other vertices; in other
words, this is site percolation on the graph $G$, regarded as a stochastic
process where $p$ increases from 0 to 1.
There is a large literature on site percolation on various finite and
infinite graphs; much of it concerns global properties, but we do not know
any references studying the local properties studied here.

Our method of analysis is based on the second version, with random times.
(The method in \cite{HPW} for random labelled trees is very different, and
is based on a remarkable exact formula for a multivariate generating function.)
The main part of our proofs are done for the case when 
the graph  $G$ is deterministic and the uncovering times are random.
By standard methods we then randomize and get results for random $G$, and
also derandomize and get results for fixed uncovering times.

Our method is a variant of methods used since a long time for the study of
\ER{} random graphs. Recall that Erd\H os and R\'enyi in their seminal
papers \cite{ER1959,ER1960} on random graphs considered the sequence of graphs
obtained 
by uncovering the edges of the complete graph $\sK_n$ in random order;
the problem studied here is thus the ``dual'' vertex analogue 
(for an arbitrary graph $G$). As is well known, it is often easier to
consider the random time version of the \ER{} random graph process, where 
edges are added (or uncovered) at independent, uniform random times.
(This process was introduced by \citet{Ste1969,Ste1970}, although there with
exponential times.) See further e.g.\ \cite[p.~4]{JLR}.
We will here use a vertex version of a method used for these random graph
processes in \cite{SJ79,SJ94}; the method is based on a martingale limit
theorem for \ctime{} martingales by \citet{JS}.

Notation and some other preliminaries are given in \refS{Snot}.
The main theorems are stated in \refS{Smain}.
A number of examples are given in \refS{Sex}, both for their own sake and to
illustrate various features of the results. 
Proofs are given in \refSs{Sprel2}--\ref{Spf2};
\refS{Sprel2} contain further preliminaries: \refS{Spf1} contains the basic
technical work including a decomposition of the \ctime{} process 
using some martingales that are defined and studied there.
The proofs are then completed in a rather straightforward manner in \refS{Spf2}.
\refS{Scomp} gives,
as a corollary, 
for the case of trees a result on the number of components in
the visible subgraph.
\refS{Ssmall} briefly discusses extensions to the number of other small
subgraphs.

The appendices give some background results used in the main part of the paper,
for which we have not found any references; the results in the appendices
are stated in rather general forms for future reference.
\refApp{AA} shows how results for the \dtime{} version can be obtained from
\ctime{}   result.
\refApp{Atrees} shows results on vertex degrees for some random trees that are
used in examples in \refS{Sex}.

\begin{ack}
  I thank Stephan Wagner for interesting discussions. 
\end{ack}

\section{Notation and preliminaries}\label{Snot}
\subsection{General (mainly standard) notation}\label{Snot1}
We denote the  size (number of elements) of a (finite) set $A$
by $|A|$.

If $H$ and $G$ are graphs, then $\hom(H,G)$ denotes the number of homomorphisms
$H\to G$, i.e., the number of (labelled, not necessarily induced) 
copies of $H$ in $G$.

%We use standard notation for some special graphs:
$\sC_n$ is a cycle with $n$ vertices, $\sP_n$ is a path with $n$ vertices,
$\sK_{\ell ,m}$ is a complete bipartite graph with $\ell+m$ vertices.

$C$ denotes unspecified constants that may vary from one occurrence to the
next. 

Unspecified limits are as \ntoo.

We use standard $O$ and $o$ notation; furthermore, $a_n\ll b_n$ means the same
as $a_n=o(b_n)$.

We use $\pto$ for convergence in probability, and $\dto$ for convergence in
distribution of random variables.
Moreover,
if $X_n$ is a sequence of random variables
and $a_n$ is a sequence of positive numbers, 
then $X_n=\op(a_n)$ means that
$X_n/a_n\pto0$ as \ntoo,
%$X_n/a_n$ converges in probability to 0, 
%\ie, if for every $\eps>0$, we have $\PP(|X_n|>\eps a_n)\to0$.
%Note that this is equivalent to the existence of a sequence $\eps_n\to0$ such
%that 
%$\PP(|X_n|>\eps_n a_n)\to0$, or in other words $|X_n|\le\eps_na_n$ w.h.p.
and
$X_n=\Op(a_n)$ means that 
%$\forall\eps>0\;\exists C\; \forall n
for every $\eps>0$, there exists $C<\infty$ such that
$\P\bigpar{|X_n|>Ca_n}<\eps$
for all $n$.
%(Equivalently, $X_n=\op(\go_na_n)$, for any sequence $\omega(n)\to\infty$.
(This is called that $X_n/a_n$ is \emph{bounded in probability} or
\emph{tight}.) 
For stochastic processes $(X_n(t))_{t\in J}$ defined on some interval $J$, we
write $X_n(t)=\opx(a_n)$ and $X_n(t)=\Opx(a_n)$ when
$\sup_{t\in J}|X_n(t)|=\op(a_n)$ and $\sup_{t\in J}|X_n(t)|=\Op(a_n)$,
respectively.

$\cL(X)$ denotes the distribution of a random variable $X$.

$\BR(t)$ denotes a Brownian bridge, 
i.e., a continuous Gaussian process on
$\oi$ with mean 0 and covariance function 
\begin{align}\label{brb}
  \Cov\bigpar{\BR(s),\BR(t)}=
s(1-t),
\qquad 0\le s\le t\le 1.
\end{align}

For typographical reasons, we write vectors as row vectors.
The transpose of $v$ is denoted $v'$.
The covariance matrix $\Cov(X)$ of a  random vector $X$, for simplicity
assumed centred, is thus 
$\E\xpar{XX'}$, and similarly
$\Cov(X,Y)=\E\xpar{XY'}$
for two centred random vectors $X$ and $Y$.

\subsection{Notation for our problem}
Let $G$ be a deterministic or random graph  
with vertex set $V(G)=[n]=\setn$ and edge set
$E=E(G)$. 
(Thus the number of vertices is $n$ and the number of edges is $|E|$.)
As usual, we denote (potential) edges by $ij$, where $i,j\in[n]$ 
(with $i\neq j$)
are the
endpoints. We sometimes write $i\sim j$ instead of $ij\in E$.

In the first version  (\dtime) of our problem,
we uncover the vertices in uniformly random order
as $v_1,\dots,v_n$; we say that vertex $v_k$ becomes visible at time $k$,
and we let $\iL_k$ be the number of edges visible at time $k$
(meaning that both endpoints are visible),
\xfootnote{To distinguish the two versions,
we use a dot in our notation  $\iL$ 
for the \dtime{} process, and also later 
for \ctime{} limits $\iZ$ of such processes, 
and other quantities related to $\iL$.
}
\ie,
\begin{align}\label{a1}
  \iL_k:=\abs{\bigset{(i,j):1\le i<j\le k\text{ and } v_iv_j\in E}},
\qquad 0\le k\le n.
\end{align}

In the second version (random times), we instead give each vertex $i$ a
random time $T_i$ when it becomes visible;
we assume that $T_1,\dots,T_n$ are independent and have the uniform
distribution  $\Uoi$. We let $L(t)$ be the number of edges visible at time
$t$, \ie,
\begin{align}\label{a2}
  L(t):=\abs{\bigset{ij\in E: T_i\le t \text{ and } T_j\le t}},
\qquad 0\le t\le 1.
\end{align}

In the case when $G$ is a random graph, we assume that the random permutation
$v_1,\dots,v_n$ and the random times $T_1,\dots,T_n$ are independent of $G$.

We note that there is a natural coupling of the two versions.
The random times $T_1,\dots,T_n$ are \as{} distinct, and we will tacitly assume
in the sequel that this is the case.
We may then let $v_k$ be the $k$th vertex that becomes visible; this yields
a uniformly random permutation of the vertices as required above.
We will always assume that we have coupled the two versions in this way.

To express this coupling in formulas, let
\begin{align}\label{nt}
  N(t):=\sumin \indic{T_i\le t},
\qquad 0\le t\le 1,
\end{align}
\ie, the number of vertices visible at time $t$.
Furthermore, let 
\begin{align}\label{tauk}
\tau_k:=\inf\set{t:N(t)\ge k} ,
\qquad k=1,\dots,n, 
\end{align}
\ie, the time
when the $k$th vertex becomes visible. 
Then
\begin{align}\label{a5}
  \iL_k = L(\tau_k).
\end{align}
We will do most of the analysis for $L(t)$, and then use \eqref{a5} to derive
corresponding results for $\iL_k$.

We introduce some further notation.
We  denote the degree of vertex $i$ by $d_i$, and 
let as usual
\begin{align}
  \gD:=\max_{1\le i\le n}d_i.
\end{align}
(Although we will for emphasis also write $\max_i d_i$ sometimes.)
We further define
\begin{align}
  \bd&:=\frac{1}{n}\sumin d_i,\label{bd}
\\\label{chi}
\chi&:=\frac{1}{n}\sumin d_i^2,
\end{align}
\ie, the first and second moments of the degree
%distribution of $G$, \ie, the $r$th moment of the degree 
of a randomly
chosen vertex in $G$.

In particular,
we note that
\begin{align}\label{di}
|E|=\frac12\sumin d_i = \frac{n\bd}2.
\end{align}

Our theorems are stated as limits for a sequence $G\nn$ of graphs as above,
with $V(G\nn)=[n]$. We then add a superscript ${}\nn$ to the notation for
all variables relating to $G\nn$; however, this may be omitted
when it is clear from the context.

\subsection{The Skorohod topology}\label{SSSkor}

We state our main results as convergence of (\ctime) stochastic processes,
defined on $\oi$.
In the proofs we will also show auxiliary results with convergence of
stochastic processes defined on the half-open interval $\oio$.
All our \ctime{} stochastic processes will be right-continuous with left
limits everywhere; such functions are often called \emph{c\`adl\`ag}.

We denote left limits by $f(t-):=\lim_{s\upto t} f(s)$,
and jump sizes by $\gD f(t):=f(t)-f(t-)$.

In general, for any interval $J\subseteq\bbR$, let $D(J)$ be the space of
{c\`adl\`ag} functions $f:J\to\bbR$.
%with left limits $f(t-)$ at all points in $J$ (except the left endpoint, if
%any). 
We equip $D(J)$, as usual, with the \emph{Skorohod topology};
a general definition is given in \cite[\S2]{SJ94} but is a bit technical,
and for our purposes it suffices to note that the topology is Polish 
(i.e., can be defined by a separable and complete metric), 
and that if $f_n,f\in D(J)$ ($n\in\bbN$) and $f$ is continuous, then
$f_n\to f$ in $D(J)$ (\ie, in the Skorohod topology) if and only if $f_n\to
f$ uniformly on every compact subset of $J$.
(All limits considered below will be continuous; thus the Skorohod topology
can be seen as a substitute for the uniform topology, which is non-separable
and has technical problems with measurability, see \cite[\S18]{Billingsley}.)
See also \eg{} \cite{Billingsley, JS, Kallenberg} for details. (These
references treat only $J=\oi$ or $J=\ooo$; the latter is equivalent to $\oio$
by a change of time).

More generally, we may also define the space $D(J)$ for vector-valued
functions; this enables us to talk about joint convergence in $D(J)$ of
several processes.

Note that convergence in $D\oi$ is substantially stronger than
convergence in $D\oio$.
We will use both. When nothing is said explicitly, 
we mean convergence in $D\oi$.

\section{Main results}\label{Smain}

We state our main results in this section. Proofs are given in \refS{Spf2}.
We use the notation in \refS{Snot}; in particular, recall that
$E\nn:=E(G\nn)$.
We state the results in three different theorems, with different conditions
on the vertex degrees. Actually, the first two theorems (\refTs{TA} and
\ref{TB})
are special cases of
the third theorem (\refT{TC}), but we have chosen to present (and prove)
them separately, in order to illustrate different features of the results
(and proofs); this also gives slightly simpler statements of the first two
theorems.

We begin with the sparse case, with $|E\nn|=O(n)$.
This includes the random labelled tree studied in \cite{HPW}.
Moreover, the sparse case is some sense the most interesting case,
where (as we will see below) 
different contributions to the result turn out to be of the same
order, and therefore interact.

\begin{theorem}\label{TA}
 Let $G\nn$ be a sequence of deterministic or random graphs with 
$V(G\nn)=[n]$.
Assume also that for some (non-random) constants $\dx,\chix\in\ooo$, 
we have, as \ntoo,
\begin{gather}
\bd\nn:=
\frac{1}{n}\sumin d\nn_i=\frac{2|E\nn|}{n}\pto\dx,\label{ta1}
\\
\chi\nn:=
\frac{1}{n}\sumin \bigpar{d\nn_i}^2\pto\chix,\label{ta2}
\\
n\qqw\max_i d\nn_i \pto0.\label{ta3}
\end{gather}
\begin{romenumerate}

\item \label{TAd}
Then, in $D\oi$,
\begin{align}\label{tad}
  n\qqw\bigpar{\iL\nn_\fnt-t^2|E\nn|}
\dto
\iZ(t),
\end{align}
where $\iZ(t)$ is a continuous Gaussian process on $\oi$ with $\E \iZ(t)=0$
and covariance function,
for $ 0\le s\le t\le1$,
\begin{align}\label{tad2}
 \Cov\bigpar{\iZ(s),\iZ(t)}
=
\igs(s,t):=
\frac{\dx}{2}{s^2}{(1-t)^2} + \gamx s^2t(1-t)
,\end{align}
where $\gamx:={\chix-\dx^2}$.

\item \label{TAc}
Similarly, in $D\oi$,
\begin{align}\label{tac}
  n\qqw\bigpar{L\nn(t)-t^2|E\nn|}
\dto
Z(t),
\end{align}
where $Z(t)$ is a continuous Gaussian process on $\oi$ with $\E Z(t)=0$
and covariance function,
for $ 0\le s\le t\le1$,
\begin{align}\label{tac2}
 \Cov\bigpar{Z(s),Z(t)}
=
\gs(s,t):=
\frac{\dx}{2}{s^2}{(1-t)^2} + \chix s^2t(1-t)
.\end{align}
\end{romenumerate}
\end{theorem}
The condition \eqref{ta3} on the maximum degree is necessary,
see \refE{Ebad1}.

\begin{remark}\label{Rdeg}
  Note that the \lhs{s} of \eqref{ta1} and \eqref{ta2} are the first and
  second moments of the degree distribution in $G\nn$; thus our
  assumptions say that these moments are asymptotically $\dx$ and $\chix$,
  respectively; as a consequence the constant $\gamx$ in \eqref{tad2}
  is the asymptotic variance of the degree distribution.
\end{remark}

A special case is when $G\nn$ is a tree (as in \cite{HPW}).
Then $|E\nn|=n-1$, and thus
\eqref{ta1} always holds with $\dx=2$, so we only have to verify
\eqref{ta2} and \eqref{ta3}; we state this as a corollary.

\begin{corollary}\label{CA}
\;  Let\/ $G\nn$ be a sequence of deterministic or random trees with\/
  $V(G\nn)=[n]$. Assume also that \eqref{ta2} and \eqref{ta3} hold.
Then \eqref{tad}--\eqref{tad2} and \eqref{tac}--\eqref{tac2} hold, with
$\dx=2$, $\gamx=\chix-4$, and $|E\nn|=n-1$ (which may be replaced by $n$).
\end{corollary}

\begin{remark}\label{R3}
We see from \eqref{tad2}, and in more detail from the proof in \refS{Spf2},
that the limit process $\iZ(t)$ in \eqref{tad} 
can be regarded as consisting of two components: 
the second term in $\igs(s,t)$ comes from the
randomness of the degrees of the vertices that are visible 
and first term comes from additional 
randomness in the structure of the visible subgraph.
In continuous time, the last term in $\gs(s,t)$ in \eqref{tac2} includes
also a term $\dx^2s^2t(1-t)$ coming
from the randomness of the number of visible vertices, which contributes a
third component to the limit.
More precisely,
the proofs show that for finite $n$, we can decompose the processes 
$\iL\nn_\fnt-t^2|E\nn|$ and $L\nn(t)-t^2|E\nn|$ into two or three components 
(+ smaller error terms)
with the origins just described.
Note that in the sparse case these three contributions to the processes 
are of the same order, unlike in other cases
discussed below; this makes the sparse case more complicated, and therefore
is a sense more interesting, than more dense cases.
\end{remark}

We next consider regular graphs. The statement below includes both the
sparse case and denser cases. 
We state the regular case separately, since
the case of regular graphs is special, and
somewhat simpler than others, because there is no randomness in the degrees
of the visible vertices, and thus one of the three contributions 
discussed in \refR{R3} disappears.
The sparse case (with the degree $d\nn$ bounded, which is essentially
equivalent to a constant degree $d\nn$) is a special case of \refT{TA}, but
the conclusions are written in a somewhat different (but equivalent) form.
In denser cases, with $d\nn\to\infty$, note that the normalizing factors in
\eqref{tbd} and \eqref{tbdii} are of different orders.
This is because in the \ctime{} version, the third contribution discussed
in \refR{R3} (which does not appear for the \dtime{} version) 
is of larger order than the others, and thus dominates the limit.
This also means that in the dense case,
the limit for the \ctime{} version (\refT{TB}\ref{TB3}) is rather
uninteresting and determined solely by the number of visible vertices. 

\begin{theorem}\label{TB}\SJ
 Let $G\nn$ be a sequence of deterministic or random graphs with 
$V(G\nn)=[n]$. %, and let $E\nn:=E(G\nn)$.
Assume also that each $G\nn$ is regular, with (non-random) degree $d\nn\ge1$,
and that $d\nn=o(n)$.
\begin{romenumerate}
\item\label{TB1}   
Then, in $D\oi$,
\begin{align}\label{tbd}
  (nd\nn)\qqw\bigpar{\iL\nn_\fnt-t^2|E\nn|}
\dto
\iZ(t),
\end{align}
where $\iZ(t)$ is a continuous Gaussian process on $\oi$ with $\E \iZ(t)=0$
and covariance function,
for $ 0\le s\le t\le1$,
\begin{align}\label{tbd2}
 \Cov\bigpar{\iZ(s),\iZ(t)}
=
\igs(s,t):=
\frac{1}{2}{s^2}{(1-t)^2} 
.\end{align}

\item \label{TB2}
If furthermore $d\nn\to \doo\le\infty$, then,
in $D\oi$,
\begin{align}\label{tbdii}
  (n\qq d\nn)\qw\bigpar{L\nn(t)-t^2|E\nn|}
\dto
Z(t),
\end{align}
where $Z(t)$ is a continuous Gaussian process on $\oi$ with $\E Z(t)=0$
and covariance function,
for $ 0\le s\le t\le1$,
\begin{align}\label{tbdii2}
 \Cov\bigpar{Z(s),Z(t)}
=
\gs(s,t):=
\frac{1}{2\doo}{s^2}{(1-t)^2} +s^2t(1-t)
.\end{align}

\item \label{TB3}
In particular, if $d\nn\to\infty$, then
\eqref{tbdii} holds with
$Z(t)=t\BR(t)$ for a Brownian bridge $\BR(t)$,
and thus,
for $ 0\le s\le t\le1$,
\begin{align}\label{tbdii3}
 \Cov\bigpar{Z(s),Z(t)}
=s^2t(1-t)
.\end{align}
\end{romenumerate}
\end{theorem}

The condition $d\nn=o(n)$ is necessary in \refT{TB}, at least for part
\ref{TB1}, 
see \refE{Ebad1}; see also \refE{Ebad2} where this condition is violated and
a non-normal limit appears.

Finally, we give a more general version.
It is easily seen that \refTs{TA} and \ref{TB} are special cases, with
$\gb_n=n\qq$ and $\gb_n=(nd\nn)\qq$, respectively.
We see also that the sizes of the first two contributions discussed in
\refR{R3} are 
governed by $\gl_1$ and $\gl_2$ in \eqref{tc1}--\eqref{tc2};
any of these may vanish (see \refE{EC}), and then only the other contributes
to the limit for the \dtime{} version.
Similarly, for the \ctime{} version, the third contribution is governed by
$\ga$.
We will see in \refE{EC} that more or less arbitrary combinations of
$\gl_1$, $\gl_2$, and $\ga$ may occur. 
(However, see \refR{Rga} below.)
Hence, different combinations of the
three components discussed in \refR{R3} may dominate in different examples.
In particular, in dense cases, 
for the \ctime{} version
we typically have $\ga=\infty$, 
and then
(\refT{TC}\ref{TCcb}) we have, as
in the regular case, a rather uninteresting limit determined by the number of
visible vertices, which dominates the contributions coming from the structure
of the visible subgraph. Here, however, the condition for this is a little
more complicated.

\begin{theorem}\label{TC}
 Let $G\nn$ be a sequence of deterministic or random graphs with 
$V(G\nn)=[n]$, and let $\gb_n$ be a sequence of positive constants
with $\gb_n=o(n)$.
Assume also that for some (non-random) constants $\gl_1,\gl_2\in\ooo$, 
we have, as \ntoo,
\begin{gather}
\frac{2|E\nn|}{\gb_n^2}=\frac{n\bd\nn}{\gb_n^2}\pto\gl_1,\label{tc1}
\\
\frac{1}{\gb_n^2}\sumin \bigpar{d\nn_i-\bd\nn}^2\pto\gl_2,\label{tc2}
\\
\gb_n\qw\max_i d\nn_i \pto0.\label{tc3}
\end{gather}
\begin{romenumerate}

\item \label{TCd}
Then, in $D\oi$,
\begin{align}\label{tcd}
\gb_n\qw\bigpar{\iL\nn_\fnt-t^2|E\nn|}
\dto
\iZ(t),
\end{align}
where $\iZ(t)$ is a continuous Gaussian process on $\oi$ with $\E \iZ(t)=0$
and covariance function,
for $ 0\le s\le t\le1$,
\begin{align}\label{tcd2}
 \Cov\bigpar{\iZ(s),\iZ(t)}
=
\igs(s,t):=
\frac{\gl_1}{2}{s^2}{(1-t)^2}+
\gl_2s^2t(1-t).
\end{align}

\item \label{TCc}
Suppose further that, for some non-random constant 
$\ga\in[0,\infty]$,
\begin{align}\label{tc4}
 n\qq\bd\nn/\gb_n\pto \ga
.\end{align}
\begin{enumerate}
\item\label{TCca} 
If\/ $0\le \ga<\infty$, 
then, in $D\oi$,
\begin{align}\label{tcc}
  \gb_n\qw\bigpar{L\nn(t)-t^2|E\nn|}
\dto
Z(t),
\end{align}
where $Z(t)$ is a continuous Gaussian process on $\oi$ with $\E Z(t)=0$
and covariance function,
for $ 0\le s\le t\le1$,
\begin{align}\label{tcc2}
 \Cov\bigpar{Z(s),Z(t)}
=
\gs(s,t):=
\frac{\gl_1}{2}{s^2}{(1-t)^2}+
(\gl_2+\ga^2)s^2t(1-t).
\end{align}
\item \label{TCcb}
If\/ $\ga=\infty$, then, 
in $D\oi$,
\begin{align}\label{tcdii}
  (n\qq \bd\nn)\qw\bigpar{L\nn(t)-t^2|E\nn|}
\dto
Z(t),
\end{align}
where 
$Z(t)=t\BR(t)$ for a Brownian bridge $\BR(t)$,
and thus, 
for $ 0\le s\le t\le1$,
\begin{align}\label{tcdii2}
 \Cov\bigpar{Z(s),Z(t)}
=
\gs(s,t):=
s^2t(1-t)
.\end{align}
\end{enumerate}
\end{romenumerate}
\end{theorem}
The condition \eqref{tc3} on the maximum degree is necessary for
\refT{TC}\ref{TCd} and \ref{TCca}, see \refE{Ebad1}.

\begin{remark}\label{Rga}
The conditions  \eqref{tc1} and \eqref{tc4} imply 
\begin{align}\label{rga}
\bd\nn\pto\ga^2/\gl_1
\end{align}
unless $\ga=\gl_1=0$.
Hence,
we can have $\gl_1>0$ and $\ga=0$ only when $\bd\nn\pto0$,
which is a rather extreme (and perhaps less interesting) case when
necessarily most vertices are isolated. 

Furthermore,  
the condition \eqref{tc2} may also be written
\begin{align}
  \label{tc2b}
\frac{1}{\gb_n^2}\Bigpar{\sumin \xpar{d\nn_i}^2-n\xpar{\bd\nn}^2}\pto\gl_2.
\end{align}
In particular,
when $\bd\nn\pto0$,  and assuming \eqref{tc1},
\eqref{tc2} is equivalent to
\begin{align}
  \label{tc2x}
\frac{1}{\gb_n^2}\sumin \xpar{d\nn_i}^2\pto\gl_2.
\end{align}
It follows that in the case $\bd\nn\pto0$, we must have $\gl_2\ge\gl_1$.
In particular, if $\gl_1>0$ and $\ga=0$, then $\gl_2\ge\gl_1>0$.
\end{remark}

\begin{remark}\label{RBrown}
  The limit processes $\iZ(t)$ and $Z(t)$ in the theorems above 
all are centred Gaussian, with covariance functions of the type
$\Cov\xpar{Z(s),Z(t)}=s^2\gf(t)$ 
($0\le s\le t\le 1$),
where $\gf(t)=a(1-t)^2+bt(1-t)$ for some $a,b\ge0$.
As noted in a special case in \cite{HPW}, this means that they can be
represented as
\begin{align}
  \gf(t)W\bigpar{t^2/\gf(t)},
\qquad 0\le t<1,
\end{align}
where $W(t)$ is a standard Brownian motion (Wiener process).
\end{remark}

\begin{remark}\label{Rinterpol}
For the \dtime{} version $\iL_k$, we thus extend it to continuous time by
considering $\iL_{\floor t}$, $t\in[0,n]$; we then scale time to $\oi$
and consider $\iL_\fnt$ in our theorems.
As is well known from many other problems,
an alternative 
would be to extend $\iL_k$ to $[0,n]$ by linear
interpolation to a continuous process $\iL_t$ as in \cite{HPW}; 
it follows immediately that the limit results above
hold also if we replace $\iL_\fnt$ by $\iL_{nt}$;
moreover, then the results
could be stated as
convergence (after rescaling) of $\iL_{nt}$
in the space $C\oi$ of continuous functions on $\oi$,
since for continuous processes,
convergence in $C\oi$ is equivalent to the convergence in $D\oi$
considered in the present paper. 
(We have chosen to use $D\oi$ and the formulations 
above with $\iL_\fnt$, at least partly because 
it is convenient to use 
discontinuous processes in our proofs.)
%See also \refE{Ebad2}.
\end{remark}

\section{Examples}\label{Sex}
In our first examples, $G\nn$ is a tree, so we can use \refC{CA}.

\begin{example}\label{Elabelledtree}
  Let us first revisit the case studied by \citet{HPW}, where $G\nn$ is a
 random labelled tree. 
It is well known that the asymptotic degree distribution is 
$1+\Po(1)$,
and it is easily seen that also all moments converge 
(given $G\nn$, in probability),
see \refR{RGWr}.
In particular, \eqref{ta2} holds with $\chix=\E(\xi+1)^2=5$, 
and \eqref{ta3} holds as a consequence of the convergence of the third moment
of the degree distribution, or by the argument in \refSS{SSGW}.
(In fact, more strongly,
$\gD\nn=\op(\log n)$, and a very precise result is known, 
see \cite{MM1991} and \cite[Remark 3.14]{Drmota}.)
See also \refE{EGW} for a generalization.
Consequently, \refC{CA} %\refT{TA} 
shows that \eqref{tad}--\eqref{tac2} hold, with
$\gamx=\chix-4=1$; thus 
the limits $Z(t)$ and $\iZ(t)$ are centred Gaussian processes with
covariance functions
\begin{align}\label{elab1}
  \Cov\bigpar{\iZ(s),\iZ(t)}&=s^2(1-t)^2+s^2t(1-t)
=s^2(1-t),
\\\label{elab2}
  \Cov\bigpar{Z(s),Z(t)}&=s^2(1-t)^2+5s^2t(1-t).
\end{align}
The limit $\iZ(t)$ with covariance function \eqref{elab1} was found by
\cite[Theorem 3]{HPW}, which inspired the present work.
\end{example}

\begin{example}\label{EGW}
  More generally, let $G\nn$ be a \cGWt{} with $n$ vertices,
defined by conditioning a \GWt{} $\cT$ with offspring
  distributed as a random variable $\xi$ with values in \set{0,1,\dots},
see \refApp{SSGW} and \eg{} \cite{SJ264}. 

Assume that $\E\xi=1$ and $0<\Var\xi<\infty$.
Then the asymptotic outdegree distribution is given by $\xi$
\cite[Theorem 7.11]{SJ264},
and the
asymptotic degree distribution is thus $1+\xi$.
We will verify in \refApp{SSGW}
that \eqref{ta2} and \eqref{ta3} hold, with 
\begin{align}\label{egw1}
\chix&=\E(\xi+1)^2=\Var\xi+4,
\end{align}
\refC{CA} thus applies and yields \eqref{tad}--\eqref{tac2}, with
\begin{align}\label{egw2}
\gamx&=\chix-4=\Var\xi.
\end{align}

Some well known examples of \cGWt{s} are
\begin{romenumerate}
\item 
The  random labelled tree in
\refE{Elabelledtree}, with
$\xi\sim\Po(1)$ and $\gamx=1$.
\item 
The
random binary tree, with 
$\xi\sim\Bin(2,\frac12)$ and $\gamx=\frac12$.
\item 
The random ordered (plane) tree, with
$\xi\sim\Ge(1/2)$ and $\gamx=2$.
\end{romenumerate}
See 
\eg{}  \cite{AldousII}, \cite{Devroye1998}, \cite{Drmota}, 
and \cite[Section 10]{SJ264},
where also further examples are given,
\end{example}

\begin{example}\label{EBST}
  Let $G\nn$ be a random binary search tree.
All outdegrees are 0, 1, or 2, and thus all degrees are 1, 2, or 3.
The proportion of vertices of each type tends in probability to $1/3$,
see \eg{} \cite[Theorem 2]{Devroye1991},
\cite[Section 3.3]{Aldous-fringe}, 
and \cite[Example 6.2]{SJ306}.
Since there is only a finite number of possible vertex degrees, 
this implies immediately that 
\eqref{ta2} holds with $\chix=(1+4+9)/3=14/3$. 
Furthermore, \eqref{ta3} is trivial.
Consequently, \refC{CA} shows that \eqref{tad}--\eqref{tac2} hold, 
with $\gamx=\chix-4=2/3$.
\end{example}

\begin{example}\label{ERRT}
Let $G\nn$ be a random recursive tree.
The asymptotic outdegree distribution is geometric $\Ge(1/2)$,
just as for random ordered trees in \refE{EGW},
see \eg{} 
\cite{MM1988}, 
\cite[Section 3.2]{Aldous-fringe}, 
\cite[Theorem 1]{SJ155},
\cite[Theorem 6.8]{Drmota},
and 
\cite[Example 6.1]{SJ306}.
Furthermore,
if $\xi\in\Ge(1/2)$, then
\eqref{ta2} and \eqref{ta3} hold, with $\chix=\E(\xi+1)^2=\Var\xi+4=6$;
see \refApp{SSRRT} for a detailed verification.
(In fact, $\gD\nn=\Op(\log n)$, see \cite[Theorem 6.12]{Drmota} for a
precise result.)
Consequently, \refC{CA} applies, with $\gamx =\Var\xi=2$,
and yields \eqref{tad}--\eqref{tac2} with exactly the same limits as for 
the random ordered tree in \refE{EGW}.
(This coincidence is thus because the  two types of random trees have the same 
asymptotic degree distribution. In other respects, the trees are quite
different.)
\end{example}

\begin{example}\label{EafH}
Let $G\nn=\sP_n$, a path of length $n$. (Thus $G\nn$ is non-random.)
This case was studied in \cite{SJ198}, in the analysis of a problem by
\cite{afH} which we briefly discuss in \refS{Scomp}.
It is obvious that \eqref{ta1}--\eqref{ta3} hold, with $\dx=2$ and
$\chix=4$.
Hence, \refT{TA} (or \refC{CA}) applies and shows \eqref{tad}--\eqref{tac2},
with  $\gamx=\chix-\dx^2=0$.

Alternatively, let $G\nn=\sC_n$, a cycle of length $n$.
This differs from $\sP_n$ only in a single edge, and thus $\iL_k$ and $L(t)$
differ by at most 1 between the two graphs; consequently, we have the same
limit results for $\sP_n$ and $\sC_n$.
Indeed,
\refT{TB} applies to $\sC_n$ with $\dx=2$, which gives the same results as
just obtained for $\sP_n$ (although written with somewhat different
normalizations). 
\end{example}

\begin{example}\label{Egnm}
Let $G\nn$ be the random graph $G(n,m_n)$, where
$m_n$ are given with $1\ll m_n \ll \binom n2$.
In this case, our problem is essentially trivial, since by definition,
$G(n,m_n)$ has exactly $m_n$ edges, which form a uniformly random subset of
size $m_n$ in the set of all $\binom n2$ possible edges.
By symmetry, we may as well uncover the vertices in order $1,2,\dots$, and
then $\iL\nn_k$ is the number of edges seen in the first $\binom k2$ possible
positions (in the order the possible edge positions are uncovered).
This number thus has a hypergeometric distribution. Moreover, 
the functional limit theorem 
for sampling from a finite population 
\cite[Theorem 24.1]{Billingsley} 
implies easily that, with
$p_n:=m_n/\binom n2\to0$, in $D\oi$,
\begin{align}\label{gnm1}
  \frac{1}{\sqrt{\binom n2 p_n(1-p_n)}}
%\Bigpar{{\binom n2 p_n(1-p_n)}}\qqw
\Bigpar{\iL\nn_\fnt - \frac{\fnt(\fnt-1)}{n(n-1)} m_n}
\dto \BR(t^2)
\end{align}
and thus
\begin{align}\label{gnm2}
  \frac{1}{\sqrt{m_n}}
\bigpar{\iL\nn_\fnt - t^2 m_n}
\dto \BR(t^2).
\end{align}
The limit process is thus a time-changed Brownian bridge.

As an illustration of our results, we show how this  also follows from
the theorems above, more precisely \refT{TC}. (If
$m_n=\Theta(n)$, we may also use \refT{TA}.)
We take $\gb_n:=\sqrt{m_n}$. We have $\bd\nn=2m_n/n$, and thus \eqref{tc1}
holds trivially with $\gl_1=2$.
We see also that \eqref{tc4} holds with $\ga=2\gl\qq$ if 
$m_n/n\to \gl\in[0,\infty]$.
We claim that \eqref{tc2} holds with $\gl_2=2$, and that \eqref{tc3} holds.
Then \refT{TC} applies, and \eqref{tcd2} yields
\begin{align}
  \igs(s,t)=s^2(1-t)^2+2s^2t(1-t)=s^2(1-t^2),
\qquad 0\le s\le t\le 1,
\end{align}
    which shows that $\iZ(t)\eqd \BR(t^2)$ (as processes). Hence, \eqref{tcd}
    yields \eqref{gnm2}. The corresponding result for $L\nn(t)$ is given by
    \eqref{tcc} or \eqref{tcdii}, depending on $\ga$.

To verify the claims, let $I_{ij}$ be the indicator that there is an edge
$ij$ in $G\nn$. Then
\begin{align}\label{gnm4}
  \sumin d_i^2=\sumin d_i +\sumin d_i(d_i-1) 
= 2m_n + \sumx_{i,j,k}I_{ij}I_{ik},
\end{align}
where $\sumx$ denotes the sum over distinct $i,j,k$.
With our choice $\gb_n=\sqrt{m_n}$,
the condition \eqref{tc2}, in the form \eqref{tc2b}, 
with $\gl_2=2$ is thus equivalent to
\begin{align}\label{gnm5}
m_n\qw\sumx_{i,j,k}I_{ij}I_{ik}-\frac{n}{m_n}(\bd\nn)^2
=
m_n\qw\sumx_{i,j,k}I_{ij}I_{ik}-\frac{4m_n}n
\pto0.  
\end{align}
Denote the triple sum in \eqref{gnm4} and \eqref{gnm5} by $V$.
Note that $V$ is twice the number of copies of $\sP_3$ in $G(n,m_n)$;
the results we need are thus closely related to result on subgraph counts
(for the special case $\sP_3$) in $G(n,m_n)$, and it seems possible 
(at least for some ranges of $m_n$) to derive
what we need from known (and more advanced) such results, 
see \eg{} 
\cite[Theorems 2a--2b]{ER1960},
\cite{Rucinski} (for $G(n,p)$) and 
\cite[Theorem 19]{SJ94},
but we find it easier to show it  directly
by calculating moments.
We have (writing $m=m_n$)
\begin{align}\label{gam6}
  \E V &= n(n-1)(n-2)\P(I_{12}I_{13}=1)
=n(n-1)(n-2)\frac{m}{\binom n2}\frac{m-1}{\binom n2 -1}
\notag\\&
=\frac{4m(m-1)}{n+1}
=\frac{4m^2}{n}+o(m).
\end{align}
A straightforward calculation, which we omit, shows that
$  \Var V = o(m^2)$. Hence,
\begin{align}\label{gam7}
V=\E V + \op(m)  = \frac{4m^2}{n}+\op(m),
\end{align}
which shows \eqref{gnm5} and thus \eqref{tc2}.
Finally, a similar calculation yields
\begin{align}
\E \sumin\Bigpar{d_i(d_i-1)-\frac{\E V}{n}}^2
=\sumin \Var\Bigpar{\sumx_{j,k}I_{ij}I_{ik}}
%=\E\sumin\Bigpar{\sumx_{j,k}\bigpar{I_{ij}I_{ik}-\E\sqpar{I_{ij}I_{ik}}}}^2
=o(m^2).
\end{align}
Hence, using also \eqref{gam6} and our assumption $m\ll\binom n2$,
\begin{align}
\gD(\gD-1)
\le\frac{\E V}{n} +
\lrpar{  \sumin\Bigpar{d_i(d_i-1)-\frac{\E V}{n}}^2}\qq
=o(m)+\op(m),
\end{align}
which yields $\gD=\op(m\qq)$ and thus \eqref{tc3}, completing the
verification that \refT{TC} applies and yields \eqref{gnm2}.
\end{example}

\begin{example}\label{Egnp}
Let $G\nn$ be the random graph $G(n,p_n)$, where
$p_n\in(0,1)$ and we assume $n^2p_n\to\infty$ and $p_n\to0$.
Thus, each possible edge appears with probability $p_n$, independently of
each other.

This example is, as the closely related preceding example,
essentially trivial.
%Since edges appear independently in $G(n,p_n)$, 
When we have uncovered $k$ vertices, the visible subgraph is a
random graph $G(k,p_n)$. Hence, $\iL\nn_k\sim\Bin\bigpar{\binom k2,p_n}$,
and, under our conditions on $p_n$, 
it follows from a version of Donsker's theorem for triangular arrays (easily
proved by \refP{P:JS} below) that
\begin{align}\label{gnp1}
\bigpar{\tbinom n2 p_n}\qqw
%\frac{1}{\binom n2 p_n}  
\Bigpar{\iL\nn_\fnt-t^2\binom{n}2p_n}\dto W(t^2)
\qquad\text{in $D\oi$},
\end{align}
where $W(t)$ is a Brownian motion on $\oi$. 
Since $E\nn=\iL\nn_n$, it follows that
\begin{align}\label{gnp2}
\bigpar{\tbinom n2 p_n}\qqw
%\frac{1}{\binom n2 p_n}  
\Bigpar{\iL\nn_\fnt-t^2|E\nn|}\dto W(t^2)-t^2 W(1) = \BR(t^2)
%\qquad\text{in $D\oi$}
.\end{align}
This follows also from \eqref{gnm2} by conditioning on $|E\nn|$.

As in \refE{Egnm}, we can also see this as an example of \refT{TC}.
If we choose $\gb_n:=(\binom n2 p_n)\qq$, then
\eqref{tc1}--\eqref{tc3} hold with $\gl_1=\gl_2=2$;
this follows from the same result for $G(n,m_n)$ in \refE{Egnm} by
conditioning on $|E\nn|$, or by similar calculations in $G(n,p_n)$.
Hence, \refT{TC} yields \eqref{gnp2}.
In the special case $p_n=\gl/n$, with $\gl\in(0,\infty)$,
we may also use \refT{TA}.
\end{example}

\begin{example}
  \label{EC}
We may construct different examples of \refT{TC} by choosing vertex degrees
$d\nn_1,\dots,d\nn_n$ and then, for example, taking $G\nn$ to be a random
graph with the given degrees. 
(As is well known, $G\nn$ can be constructed by the configuration model,
conditioned to be simple. Of course, we have to choose the degrees such that
a simple graph $G\nn$ exists; in particular we need $\sum_i d\nn_i$ to be even.)

Note that examples of \refT{TC} with $\gl_1>0$, $\gl_2=0$, 
and either $0<\ga<\infty$ or $\ga=\infty$
(see \eqref{rga})
are provided by regular graphs as in \refT{TB},  
and that examples with $\gl_1>0$, $\gl_2>0$, and $0<\ga<\infty$ 
are provided by many instances of
\refT{TA}, see the examples above.

Another interesting case is to choose the degrees such that the average degree
is roughly constant, more precisely $\bd\nn\to\dx$ for some $\dx\in(0,\infty)$,
but the variance of the degrees
\begin{align}
  \gamm\nn:=\frac{1}{n}\sumin \bigpar{d\nn_i-\bd\nn}^2
=\chi\nn-(\bd\nn)^2 \to\infty.
\end{align}
%Assume also that $\max_i(d\nn_i)^2\ll\sumin (d\nn_i)^2$.
For example, let $\gd_n\to0$ with $n\gd_n\to\infty$, and let all degrees
$d\nn_i$ be 2 except the first $n\gd_n+O(1)$ which are $\floor{1/\gd_n}$.
Then, choosing $\gb_n:=\sqrt{n\gamm\nn}$,
\refT{TC} holds with $\gl_1=0$, $\gl_2=1$, and $\ga=0$. Hence, in this case
the second component in \refR{R3} dominates both the others.
As a result, we have the same covariance in \eqref{tcd2} and \eqref{tcc2},
and thus the same limit in \eqref{tcd} and \eqref{tcc}.

We may also construct an example with $\gl_1>0$ and $\ga=0$, 
which by \refR{Rga} implies $\bd\nn\pto0$ and $\gl_2\ge\gl_1>0$.
We may simply let $G_n$ by the cycle $\sC_{m_n}$ plus $n-m_n$ isolated vertices,
where $m_n\to\infty$ with $m_n=o(n)$, and take $\gb_n:=\sqrt{m_n}$; we omit
the simple verifications.

Finally, for given sequences $a_n$ and $b_n$ of positive integers with
$b_n\le a_n$,
let $n$ be even and let $n/2$ vertices have degree $a_n+b_n$ and the other $n/2$
degree $a_n-b_n$.
Choose $\gb_n:=\sqrt n b_n$; thus \eqref{tc2} holds with $\gl_2=1$.
Let $b_n$ be any sequence such that 
$1\ll b_n \ll \sqrt n$; it is the easy to verify
that in the following three cases,
the conditions \eqref{tc1}--\eqref{tc3} and \eqref{tc4} hold with the
stated parameters:
\begin{romenumerate}
 
\item 
If $a_n=b_n^2$, then $\gl_1=1$, $\gl_2=1$, and  $\ga=\infty$.
\item 
If $a_n=b_n$, then $\gl_1=0$, $\gl_2=1$, and $\ga=1$.
\item 
If $b_n \ll a_n \ll b_n^2$, then $\gl_1=0$, $\gl_2=1$, and $\ga=\infty$.
\end{romenumerate}
\end{example}

\begin{example}\label{Ebad1}
  It is easy to see that condition \eqref{ta3} on the maximum degree
is necessary in \refT{TA}.
In fact, suppose that for some fixed $\gd>0$, there is a vertex, say $i=1$,
with $d_1\ge \gd n\qq$. Since vertex 1 and its neighbours become visible in
uniformly random order, $\gD L(T_1)$ is uniformly distributed on
\set{0,\dots,d_1}; hence, with probability $\ge 1/2$, we have
$\gD L(T_1) \ge d_1/2\ge (\gd/2) n\qq$.
Consequently, $n\qqw L(t)$ has with large probability a macroscopic jump 
(at $T_1$), and thus it cannot converge in distribution to a continuous
process. The same applies to $n\qqw\iL_k$.

The same argument shows that $d\nn=o(n)$ is necessary for \eqref{tbd} in
\refT{TB}, and that \eqref{tc3} is necessary for \eqref{tcd} and \eqref{tcc}
in \refT{TC}.
\end{example}

\begin{remark}
In spite of \refE{Ebad1},
it is for some graphs $G\nn$
possible to obtain a continuous limit in \eqref{tbd}
or \eqref{tcd} even if $d\nn=o(n)$ or $\gD\nn=o(\gb_n)$ fails,
provided we use
linear interpolation of $\iL_k$ as in \refR{Rinterpol} instead of
$\iL_\fnt$. We give one example in \refE{Ebad2}, and note that the limit
obtained there is \emph{not} Gaussian. 
We have not investigated this possibility any further, but it seems that it
reguires $G\nn$ to be regular or almost regular (in a suitable sense);
moreover, we conjecture that limits always will be as in \refE{Ebad2}, and
thus not normal.
\end{remark}

\begin{example}\label{Ebad2}
Let $G\nn=\sK_{n/2,n/2}$, the symmetric complete bipartite graph, where we
assume that $n$ is even. Then $G\nn$ is regular with degree $d\nn=n/2$.
Note that this example does \emph{not} satisfy the condition $d\nn=o(n)$ in
\refT{TB}.
We will show that, indeed, \refT{TC} does not hold, and that we in this
example have non-normal limit distributions.

Colour the vertices of the two parts white and black, respectively; each
edge has thus one white and one black endpoint.

Let $\iW_k$ be the number of uncovered white vertices among the first $k$
uncovered vertices, and let $\iY_k:=\iW_k-k/2$.
Thus, at time $k$ there are $\iW_k=k/2+\iY_k$ visible white vertices
and $\iB_k=k-\iW_k=k/2-\iY_k$ visible black vertices; consequently,
\begin{align}\label{kk1}
  \iL_k=(\tfrac{k}{2}+\iY_k)(\tfrac{k}{2}-\iY_k)
=\tfrac{k^2}{4}-\tY_k^2.
\end{align}

The random variable $\iW_k$ has a hypergeometric distribution.
Moreover, it is well known 
(\eg{} by \cite[Theorem 24.1]{Billingsley})
that, in $D\oi$, as \ntoo,
\begin{align}\label{kk2}
n\qqw(\iW_\fnt-\iB_\fnt)
 = 2n\qqw\iY_\fnt
\dto \BR(t),
\end{align}
where $\BR(t)$  is a Brownian bridge, 
see \refSS{Snot1}.
Consequently, we see from \eqref{kk1} and \eqref{kk2} that, in $D\oi$,
\begin{align}\label{kk3}
n\qw\Bigpar{\iL\nn_\fnt-\frac{\fnt^2}{4}}
=-n\qw(\iY\nn_\fnt)^2
\dto -\frac{1}4\BR(t)^2.
\end{align}
In this case, thus the limit process is a (negative) square of a Gaussian
process, and for a fixed $t\in(0,1)$, the distribution of $\iL_\fnt$ is, after
normalization and change of sign, a $\chi^2$-distribution $\chi^2(1)$.
In particular, the limit distribution is not normal.

Note also that we in \eqref{kk3}
cannot replace $\fnt^2/4$ by $(nt)^2/4=t^2|E\nn|$, as we have in
\eqref{tbd};
the reason is that the jumps in $\iL\nn_k$ and in $\fnt^2$ are of order $n$,
and do not disappear asymptotically with the normalization in \eqref{kk3};
in \eqref{kk3} the jumps of the two terms cancel asymptotically, but we
cannot replace one of the terms with a continuous version.
However, if we
as in \refR{Rinterpol} define
$\iL\nn_t$  for real $t\in[0,n]$ by linear interpolation
between integers, and thus $\iL\nn_t$ is a continuous stochastic process, then
it follows easily from \eqref{kk3} that
\begin{align}
n\qw\bigpar{\iL\nn_{nt}-t^2|E\nn|}
=
n\qw\Bigpar{\iL\nn_{nt}-\frac{n^2t^2}{4}}
\dto 
-\frac{1}4\BR(t)^2
\end{align}
in $D\oi$ (and in $C\oi$, since here all processes are continuous).
\end{example}

\section{Preliminaries}\label{Sprel2}
\subsection{Addition in the Skorohod topology}
Addition is not continuous in $D(J)$ in general, but if
$f_n,f,g_n,g\in D(J)$ with $f_n\to f$ and $g_n\to g$, and furthermore $f$
and $g$ are continuous, then $f_n+g_n\to f+g$.
(This follows immediately from the description in \refSS{SSSkor}.)
As a consequence, we have the following results, 
which often will be used without comment.

\begin{lemma}\label{Ladd}
Let $X$, $Y$, $X_n$ and $Y_n$ ($n\ge1$) be stochastic processes on an interval
$J$, with $X$ and\/ $Y$ continuous a.s.
If\/ $(X_n,Y_n)\dto (X,Y)$  in $D(J)$, 
then $X_n+Y_n\dto X+Y$ in $D(J)$.
\end{lemma}

\begin{proof}
By  the comment above and \cite[Corollary 1, p.~31]{Billingsley}.
\end{proof}

\begin{lemma}\label{L0}
Let $X$, $X_n$ and $Y_n$ ($n\ge1$) be stochastic processes on an interval
$J$, with $X$ continuous a.s.,
and suppose that $X_n\dto X$ in $D(J)$ and $Y_n=\opx(1)$.
Then $X_n+Y_n\dto X$ in $D(J)$.
\end{lemma}

\begin{proof}
Recall that $Y_n=\opx(1)$ means $\sup_{t\in J}|Y_n(t)|\pto0$.
Hence, 
$Y_n\pto 0$ in $D(J)$, and thus
$(X_n,Y_n)\dto (X,0)$
\cite[Theorem 4.4]{Billingsley}.
Consequently, the result follows from \refL{Ladd}.
\end{proof}
%If $X$ is continuous \as{} (the only case that we use),
%the result now follows, because the map $(f,g)\to f+g$ from $D(J)\times
%D(J)\to D(J)$ is continuous at pairs of continuous functions.
%
%In general, we may use 
%the Skorohod coupling theorem
%%\cite[Theorem~4.30]{Kallenberg}
%again; we may thus assume that $(X_n,Y_n^*)\asto (X,0)$, %in $D(J)\times\bbR$. 
%and it is easily seen that this implies $X_n+Y_n\to X$ also for
%discontinuous $X$; we omit the details.

\begin{remark}
  \refL{Ladd} actually holds assuming only that either $X$ or $Y$ is
  continuous a.s.; similarly, continuity of $X$ is not needed in
  \refL{L0}. We omit the proofs, since we need only the cases above.
\end{remark}

\subsection{Quadratic variation of martingales}

Let $M_t$ be a \ctime{} martingale, defined for $t\in J$ where $J$ is
some interval $[0,b]$ or $[0,b)$ with $0<b\le\infty$.
Assume for convenience $M(0)=0$.

Let $\gD M(t):=M(t) - M(t-)$ be the size of the jump (if any) at $t$.
(For $t\in J$, where we for completeness define $\gD M(0):=0$.)

We will use the \emph{quadratic variation}   $[M,M]_t$ of a \ctime{}
martingale $M$,
and its bilinear version, 
the \emph{quadratic covariation} $[M,M_1]_t$ of two martingales $M$ and $M_1$.
For a general definition 
(which further extends beyond martingales to semimartingales)
see  \eg{} \cite[\S I.4e]{JS} or 
\cite[p.~519]{Kallenberg},
%\cite[\S 2.6]{Protter},
but we only need a simple case:
If the martingale $M(t)$ \as{}
has finite variation over every compact subinterval of $J$ (this holds
trivially for the martingales defined in \eqref{QQQt}--\eqref{NNNt} below, 
since they are piecewise smooth), then its 
quadratic variation is given by, see \eg{}
\cite[Theorem 26.6(viii)]{Kallenberg},
%%\cite[pp.~70--71, in particular Theorem II.26]{Protter},
\begin{align}\label{qm1}
  [M,M]_t := \sum_{0\le s\le t}(\gD M(s))^2,
\end{align}
and similarly, for any martingale $\MMM$ on $J$,
\begin{align}\label{qm2}
  [M,\MMM]_t:= \sum_{0\le s\le t}\gD M(s) \gD \MMM(s).
\end{align}
(The sums are formally uncountable, but there is only a countable number of
non-zero terms, since $M(s)$ has at most countably many jumps.)

We recall the basic identity \cite[p.~73, Corollary 3]{Protter}
\begin{align}\label{qm3}
    \E[M,M]_t = \E\abs{M(t)}^2
\end{align}
and, more generally (by polarization), 
provided $\E |M(t)|^2<\infty$ and $\E |\MMM(t)|^2<\infty$,
\begin{align}\label{qm4}
\E{}[M,\MMM]_t:= \E\bigpar{M(t)\MMM(t)}.
\end{align}

\subsection{A martingale convergence theorem}

Our proofs are based on the following convergence theorem for martingales;
it is a special case of a more general theorem by \citet[Theorem VIII.3.12]{JS},
and the present formulation is taken from \cite[Theorem 0]{SJ79} and
\cite[Proposition 2.6]{SJ94}, where proofs are given.
(See also \cite[Proposition 9.1]{SJ154} for a similar version with somewhat
weaker assumptions.) %reused in SJ196, Proposition 4.1 (better) and SJ218,
%Prop 5.5.

\begin{proposition}\label{P:JS}
Let $J$ be an interval $[0,b]$ or $[0,b)$, $0<b\le\infty$.
Assume that for each $n$,
$M\nn(t)=(M\nn_{i}(t))_{i=1}^q$ is a $q$-dimensional martingale on $J$
with $M\nn(0)=0$, and that 
$\Sigma(t)=\xpar{\gs_{ij}(t)}_{i,j=1}^q$, $t\in J$, 
is a  \rompar(non-random) 
continuous matrix-valued function 
such that
for every fixed $t\in J$ and $1\le i,j\le q$, 
we have
\begin{align}\label{pjs1}
\E{} \bigsqpar{M\nn_i, M\nn_j}_t&\to \sigma_{ij}(t),
\\\label{pjs2}
 \Var  [M\nn_i,M\nn_j]\strut_t&\to0.
\end{align}
Then $M\nn\dto M$ in $D(J)$ as $n\to\infty$, 
where
$M$ is a continuous $q$-dimensional Gaussian process with $\E M(t)=0$ and
covariance function
\begin{align}
\Cov\bigpar{M_{i}(s),M_j(t)}
=  \E \bigpar{M_{i}(s)M_j(t)}=\sigma_{ij} (s),
\qquad  s,t\in J \text{ and } s\le t
.\end{align}
\end{proposition}
(Furthermore, the limit process $M(t)$ is a martingale, 
but we have no use for this extra  property in the present paper.)
Note that by \eqref{qm4}, \eqref{pjs1} can equivalently be
written as
$\E \bigpar{M\nn_i(t) M\nn_j(t)}\to \sigma_{ij}(t)$.

In our applications the martingales will blow up at $t=1$, making it
impossible to use \refP{P:JS} directly on the closed interval $\oi$; 
instead we will use \refP{P:JS} on $\oio$, and then
obtain convergence
on $\oi$ using the following  lemma.

\begin{lemma}\label{Loi}
Suppose that $M\nn(t)$, $n\ge1$, are martingales on $\oio$ such that
$M\nn(t)\dto M(t)$ in $D\oio$ for some continuous stochastic process $M(t)$.
Suppose furthermore that
\begin{align}\label{loi1}
\E \bigabs{M\nn(t)}^2 \le C (1-t)^{-a}
\end{align}
for some $a\ge0$, uniformly in $n\ge1$ and $t\in\oio$. 
Let $b>a/2$ and define $\tM\nn(t):=(1-t)^bM\nn(t)$
and $\tM(t):=(1-t)^bM(t)$ for $t\in\oio$, and $\tM\nn(1)=\tM(1):=0$.
Then, \as, $\tM\nn(1-)=\tM(1-)=0$, and thus \as{} $\tM\nn\in D\oi$ and
$\tM$ is continuous on $\oi$; moreover, 
\begin{align}\label{loi2}
  \tM\nn(t)\dto  \tM(t)
\end{align}
in $D\oi$.
\end{lemma}

\begin{proof}
  First, for any $N\ge1$,  Doob's inequality 
\cite[Proposition 7.16]{Kallenberg}
and \eqref{loi1} imply
  \begin{align}\label{loi3}
 \E \sup_{1-2^{-N}\le t\le 1-2^{-N-1}}|\tM\nn(t)|^2&
\le 2^{-2bN} \E \sup_{0\le t\le 1-2^{-N-1}}|M\nn(t)|^2
\notag\\&
\le 2^{-2bN} 4\E |M\nn(1-2^{-N-1})|^2
\notag\\&
\le C 2^{aN-2bN}
=C 2^{-(2b-a)N}
,
  \end{align}
and thus
  \begin{align}\label{loi4}
 \E \sup_{1-2^{-N}\le t<1}|\tM\nn(t)|^2&
\le C\sum_{\ell=N}^\infty 2^{-(2b-a)\ell}
\le C 2^{-(2b-a)N}.
  \end{align}
Letting $N\to\infty$ yields, by Fatou's lemma,
%(or dominated convergence),
  \begin{align}\label{loi5}
\E\limsup_{t\upto1}|\tM\nn(t)|^2&
= \E \lim_{\Ntoo}\sup_{1-2^{-N}\le t<1}|\tM\nn(t)|^2
\le C \lim_{\Ntoo}2^{-(2b-a)N}
\notag\\&=0,
  \end{align}
and thus $\tM\nn(1-):=\lim_{t\upto1}\tM\nn(t)=0$ \as{} as asserted.
Moreover, since \eqref{loi3} holds uniformly in $n$, it follows (by Fatou's
lemma again) that it holds for $\tM(t)$ too, and thus the same argument
shows that \as{} $\tM(1-)=0$, and thus $\tM$ is continuous.

Finally, \eqref{loi4} and Markov's inequality imply that, for any $\eps>0$
and $u\in\oio$,
\begin{align}\label{loi6}
  \sup_n \P\bigpar{\sup_{u\le t<1}|\tM\nn(t)|>\eps}\le C\eps\qww (1-u)^{2b-a},
\end{align}
which tends to 0 as $u\upto1$; hence \cite[Proposition 2.4]{SJ94} applies
and
yields $\tM\nn\dto\tM$ in $D\oi$.
\end{proof}

\begin{remark}\label{Roi}
  \refL{Loi} extends to vector-valued martingales $(M\nn_i(t))_{i=1}^q$,
where we may have separate exponents $a_i$ and $b_i$ (with $b_i>a_i/2\ge0$)
for different components $M\nn_i$, $i=1,\dots,q$.
(Thus, $\tM_i(t):=(1-t)^{b_i}M_i(t)$.)
To see this, note that the proof above shows \eqref{loi6} for each component
$\tM_i\nn$, and it follows that \eqref{loi6} holds for $\tM\nn$, with $2b-a$
replaced by $\min_i(2b_i-a_i)>0$. This and the  convergence $\tM\nn(t)\dto
M\nn(t)$ in $D\oio$ then implies convergence in $D\oi$ just as in the
1-dimensional case in
\cite[Proposition 2.4]{SJ94}.
(For a proof, use \eg{} the Skorohod coupling theorem
\cite[Theorem~4.30]{Kallenberg}; we omit the details.) 
\end{remark}

\section{A decomposition into martingales}\label{Spf1}
The main idea of the proofs is to decompose $L(t)$ as a linear combination
(with coefficients that are deterministic functions of $t$)
of some martingales defined below, and then to show (joint) convergence of
these martingales. % (using \refP{P:JS}).

In this section
we consider a fixed $n$ and construct the martingales and the decomposition
that we use; we also calculate the quadratic (co)variations of the
martingales
and make some estimates of them. 
In \refS{Spf2}, we then let \ntoo{} and show the desired convergence.

We thus assume throughout this section that $G$ is a fixed, deterministic
graph on $[n]$. Thus $E=E(G)$ is a fixed set, and the vertex degrees $d_i$ are
deterministic. 

As said in the introduction, we use a vertex version of the method in
\cite{SJ79,SJ94}.
We define, for $i\in[n]$,
the random function
\begin{align}\label{I}
  I_i(t):=\indic{T_i\le t},
\qquad 0\le t\le 1.
\end{align}
Thus $I_i(t)$ is the indicator of the event that vertex $i$ is visible at
time $t$.
We define further
\begin{align}\label{II}
  \II_i(t)&:=I_i(t)-\E I_i(t)
=I_i(t)-t,
\qquad 0\le t\le 1,
\\\label{III}
  \III_i(t)&:=(1-t)\qw\II_i(t)
=\frac{I_i(t)-t}{1-t},
\qquad 0\le t< 1.
\end{align}
Note that $\II_i(0)=\II_i(1)=0$, and that
\begin{align}\label{eiii}
\E \III_i(t)=\E\II_i(t)=0.  
\end{align}
Furthermore,
\begin{align}\label{e2iii}
  \E\III_i(t)^2=(1-t)\qww\E\II_i(t)^2
=(1-t)\qww\Var I_i(t)
=\frac{t}{1-t},
\end{align}
and, when $i\neq j$, by independence,
\begin{align}\label{coviii}
  \E\bigpar{\III_i(t)\III_j(t)}
=  \E{\III_i(t)}  \E{\III_j(t)}
=0.
\end{align}

Let $\cF_t$ be the \gsf{} generated by $\set{I_i(s):i\in[n]\text{ and }s\le t}$.
The martingales below are
martingales
with respect to the filtration $(\cF_t)_{t}$.
%, even if we do not say so explicitly.

\begin{lemma}\label{LM}
  $\III_i(t)$ is a martingale for $t\in\oio$, for every $i\in[n]$.

More generally, for any sequence $1\le i_1<\dots<i_r\le n$,
the product $\prod_{j=1}^r \III_{i_j}(t)$ is a martingale on $\oio$.
\end{lemma}

\begin{proof}
(After \cite[Lemma 2.1]{SJ94}; see also \cite[Lemma 2.1]{SJ79}.)
It is easy to see that $\III_i$ is a Markov process with
$\E\bigpar{\III_i(t)\mid\cF_s}=\III_i(s)$ when $0\le s\le t<1$,
which implies that $I_i(t)$ is a martingale. 

The final sentence follows because
the collections of random variable $\set{\III_i(t)}_{t\in\oi}$, $i\in[n]$,
are independent of each other.
\end{proof}

We have by \eqref{a2} and \eqref{I}--\eqref{II}
(summing over unordered pairs $ij$)
\begin{align}\label{b1}
  L(t)&=\sum_{ij\in E} I_i(t)I_j(t)
=\sum_{ij\in E} \bigpar{\II_i(t)+t}\bigpar{\II_j(t)+t}
\notag\\&
=\sum_{ij\in E} \II_i(t)\II_j(t)
+\sumin \sum_{j:ij\in E}\II_i(t)t
+\sum_{ij\in E} t^2
\notag\\&
=
\sum_{ij\in E} \II_i(t)\II_j(t)
+t\sumin d_i\II_i(t)+t^2|E|.
\end{align}

We define, for $t\in\oi$,
\begin{align}\label{QQt}
  \QQ(t)&:=\sum_{ij\in E}\II_i(t)\II_j(t),
\\\label{SSt}
\SS(t)&:=\sumin d_i\II_i(t),
\\\label{NNt}
\NN(t)&:=\sumin\II_i(t)=\sumin I_i(t)-nt=N(t)-nt,
\end{align}
and further, for $t\in\oio$,
recalling \eqref{III},
\begin{align}\label{QQQt}
  \QQQ(t)&:=\sum_{ij\in E}\III_i(t)\III_j(t) = (1-t)\qww\QQ(t),
\\\label{SSSt}
\SSS(t)&:=\sumin d_i\III_i(t) =(1-t)\qw\SS(t),
\\\label{NNNt}
\NNN(t)&:=\sumin\III_i(t)=(1-t)\qw\NN(t).
\end{align}
By \refL{LM}, $\QQQ(t)$, $\SSS(t)$ and $\NNN(t)$ are martingales on $\oio$.

We can now rewrite \eqref{b1} as
\begin{align}\label{b2a}
  L(t) &=  \QQ(t) + t \SS(t) + t^2|E|
\\&
= (1-t)^2 \QQQ(t) + t(1-t) \SSS(t) + t^2|E|,\label{b2b}
\end{align}
where the second line is meaningful only for $t\in[0,1)$.

The main idea in our proofs is to 
use the decomposition \eqref{b2a}--\eqref{b2b} 
together with limit theorems for $\QQ$ and $\ZZ$ 
(and $\NN$, for reasons that will be seen later),
or (essentially equivalently) for the martingales
$\QQQ$, $\SSS$ and $\NNN$,
which we obtain from \refP{P:JS} 
and calculations of quadratic (co)variations.

\subsection{Quadratic variations}
To find the quadratic (co)variations,
we note that $I_i(t)$ and $\II_i(t)$ have jumps $+1$ at $t=T_i$; hence
$\III_i(t)$ has a jump $(1-T_i)\qw$ at $T_i$ (and no other jump).
It follows from \eqref{QQQt}--\eqref{NNNt} that $\QQQ(t)$, $\SSS(t)$ and $\NNN(t)$
have jumps only at the points $T_i$, $i=1,\dots,n$, 
and that
\begin{align}
  \gD \QQQ(T_i) &=(1-T_i)\qw\sum_{j\sim i}\III_j(T_i),
\\
\gD \SSS(T_i) &= d_i(1-T_i)\qw,
\\
\gD \NNN(T_i) &= (1-T_i)\qw.
\end{align}
Hence, \eqref{qm1}--\eqref{qm2} yield
\begin{align}\label{qQQ}
[\QQQ,\QQQ]_t &= 
 \sumin \indic{T_i\le t} (1-T_i)\qww \Bigpar{\sum_{j\sim i}\III_j(T_i)}^2,
\\\label{qSS}
[\SSS,\SSS]_t &= \sumin \indic{T_i\le t}d_i^2 (1-T_i)\qww,
\\  \label{qNN}
[\NNN,\NNN]_t &= \sumin \indic{T_i\le t} (1-T_i)\qww.
\\\label{qQS}
[\QQQ,\SSS]_t&=\sumin \indic{T_i\le t} (1-T_i)\qww d_i{\sum_{j\sim i}\III_j(T_i)},
\\\label{qQN}
[\QQQ,\NNN]_t&=\sumin \indic{T_i\le t} (1-T_i)\qww {\sum_{j\sim i}\III_j(T_i)},
\\\label{qSN}
[\SSS,\NNN]_t &= \sumin \indic{T_i\le t}d_i (1-T_i)\qww.
\end{align}

We easily calculate the expectations. Since the $T_i\in\Uoi$ are independent,
we obtain from \eqref{qQQ}--\eqref{qSN}, recalling \eqref{eiii}--\eqref{coviii},
\begin{align}\label{eQQ}
\E [\QQQ,\QQQ]_t &= 
 \sumin \intot (1-s)\qww \E\Bigpar{\sum_{j\sim i}\III_j(s)}^2\dd s
=
 \sumin \intot (1-s)\qww \sum_{j\sim i}\E\III_j(s)^2\dd s
\notag\\&
=\sumin \intot d_i \frac{s}{(1-s)^3}\dd s
=\sumin d_i\cdot \frac12\frac{t^2}{(1-t)^2}
=|E|\frac{t^2}{(1-t)^2}.
%\end{align}
%Similarly, %omitting some details,
%\begin{align}
\\\label{eSS}
\E [\SSS,\SSS]_t &= \sumin \intot d_i^2 (1-s)\qww\dd s
=\sumin d_i^2\cdot \frac{t}{1-t},
\\\label{eNN}
\E [\NNN,\NNN]_t &= \sumin \intot  (1-s)\qww\dd s
=n\frac{t}{1-t},
\\\label{eQS}
\E [\QQQ,\SSS]_t &= 
 \sumin \intot (1-s)\qww d_i\E{\sum_{j\sim i}\III_j(s)}\dd s
=0,
\\\label{eQN}
\E [\QQQ,\NNN]_t &= 
 \sumin \intot (1-s)\qww \E{\sum_{j\sim i}\III_j(s)}\dd s
=0,
\\\label{eSN}
\E [\SSS,\NNN]_t &= \sumin \intot d_i (1-s)\qww\dd s
=\sumin d_i\cdot \frac{t}{1-t}
=2|E|\frac{t}{1-t}.
\end{align}

We also need estimates of the variances of the quadratic (co)variations.
(We do not bother to calculate the variances exactly, although this clearly
can be done.)
For simplicity, we consider a fixed $t$. (It is easily seen that the
constants below can be taken bounded for $t\in[0,t_0]$ for any $t_0<1$, but
that they blow up as $t\to1$.)

Recall that
$\hom(\sC_4,G)$ is the number of (labelled) copies of $\sC_4$ in $G$.

\begin{lemma}\label{LV}
  For each fixed $t\in\oio$, we have, with constants $C$ that depend on $t$,
\begin{align}\label{vQQ}
\Var [\QQQ,\QQQ]_t &
%\le C\sumin d_i^2  + C\hom(\sC_4,G)
\le
C\sumin d_i^3.
\\\label{vSS}
\Var [\SSS,\SSS]_t &
\le C \sumin d_i^4,
\\\label{vNN}
\Var [\NNN,\NNN]_t &
\le Cn,
\\\label{vQS}
\Var [\QQQ,\SSS]_t &
\le C \sumin d_i^4,
\\\label{vQN}
\Var [\QQQ,\NNN]_t &
\le C \sumin d_i^2,
\\\label{vSN}
\Var [\SSS,\NNN]_t &
\le C \sumin d_i^2.
  \end{align}
\end{lemma}

\begin{proof}
To begin with the simplest case, \eqref{qNN} shows that $[\NNN,\NNN]_t$ is the
sum of $n$ independent random variables, each bounded by $(1-t)\qww$. 
Hence each term has variance $\le (1-t)^{-4}$, and thus 
$\Var[\NNN,\NNN]_t\le n(1-t)^{-4}$, which we simplify to \eqref{vNN}.

The same argument also gives \eqref{vSS} and \eqref{vSN}.

Next, we write \eqref{qQQ} as
\begin{align}\label{el1}
  [\QQQ,\QQQ]_t = \sumin\sum_{j\sim i} A_{ij}+\sum_{i,j,k:j\sim i\sim k}B_{ijk},
\end{align}
where in the second sum we assume $j\neq k$, and we let
\begin{align}\label{elab}
  A_{ij}&:=\indic{T_i\le t} (1-T_i)\qww \III_j(T_i)^2,
\\
  B_{ijk}&:=\indic{T_i\le t} (1-T_i)\qww \III_j(T_i)\III_k(T_i) \indic{j\neq k}.
\end{align}
We estimate the variances of the two sums separately. Note that 
(for a fixed $t$)
all $A_{ij}$
and $B_{ijk}$ are uniformly bounded.
Two variables $A_{ij}$ and $A_{i'j'}$ are independent unless they have at
least one common index. Hence, using symmetry,
\begin{align}\label{ela}
  \Var\Bigpar{\sum_{i,j:i\sim j}A_{ij}}
&\le C 
\bigabs{
\set{(i,j,i',j')\in[n]^4:i\sim j,\, i'\sim j',\, \set{i,j}\cap\set{i',j'}\ge1}}
\notag\\&
\le C 
\bigabs{\set{(i,j,j')\in[n]^3:i\sim j,\, i\sim j'}}
\notag\\&
=C\sumin d_i^2.
\end{align}
For $B_{ijk}$, we first note that $\E\xpar{B_{ijk}\mid T_i,T_j}=0$ since 
$\E\xpar{\III_k(T_i)\mid T_i}=0$ by  \eqref{eiii};
thus $\E B_{ijk}=0$. 
Similarly, if, say, $k\notin\set{i',j',k'}$, then, by conditioning on
$T_\ell$ for all $\ell\neq k$, we have
$\E \xpar{B_{ijk}B_{i'j'k'}}=0$.
Hence, if $\E \xpar{B_{ijk}B_{i'j'k'}}\neq0$, then each of
$j,k,j',k'$ equals one of the other five indices. Since we assume that
$i,j,k$ are distinct, as well as $i',j',k'$, this is possible only if either
$\set{i,j,k}=\set{i',j',k'}$ or $\bigabs{\set{i,j,k}\cap\set{i',j',k'}}=2$,
and in the latter case furthermore
$\set{i,j,k}\cap\set{i',j',k'}=\set{j,k}=\set{j',k'}$, and thus $j,i,k,i'$
form a cycle $\sC_4$ in $G$. In the first case, there are at most 6 choices
of $i',j',k'$ for each $(i,j,k)$. Consequently, we obtain, using \eqref{ela},
\begin{align}\label{elb}
  \Var\Bigpar{\sum_{i,j,k:k\sim i\sim j}B_{ijk}}
&\le C 
\bigabs{\set{(i,j,k)\in[n]^3:k\sim i\sim j}}
+ C\hom(\sC_4,G)
\notag\\&
=C\sumin d_i^2  + C\hom(\sC_4,G)
.\end{align}

It follows from \eqref{el1}, \eqref{ela} and \eqref{elb} that
\begin{align}
  \Var [\QQQ,\QQQ]_t &
\le C\sumin d_i^2  + C\hom(\sC_4,G).
\end{align}
The inequality \eqref{vQQ} then follows because by an inequality by
\citet{Sidorenko} (see also \cite{SJ370}), 
\begin{align}
  \hom(\sC_4,G) \le   \hom(\sP_4,G)
\le  \hom(\sK_{1,3},G) \le \sumin d_i^3.
\end{align}
Alternatively, this follows because we have
\begin{align}\label{salt}
  \hom(\sC_4,G) &
\le   \hom(\sP_4,G)
\le  \sum_{i,j:\,i\sim j} d_i d_j
\le  \sum_{i,j:\,i\sim j} \tfrac12\bigpar{d_i^2+ d_j^2}
= \sum_{i,j:\,i\sim j} d_i^2
\notag\\&
= \sumin d_i^3.
\end{align}

Similarly, we write \eqref{qQN} as 
\begin{align}
  \label{luc1}
[\QQQ,\NNN]_t=\sumin\sum_{j\sim i} D_{ij},
\end{align}
where
\begin{align}
  \label{dij}
D_{ij}:= \indic{T_i\le t} (1-T_i)\qww \III_j(T_i)
,\end{align}
noting that
(for a given $t$), the random variables $D_{ij}$
are uniformly bounded; furthermore, 
$D_{ij}$ and $D_{i'j'}$ are independent unless they have at least one common
index.
Consequently, as in \eqref{ela},
\begin{align}\label{el7}
\Var [\QQQ,\NNN]_t = 
  \Var\Bigpar{\sum_{i,j:i\sim j}D_{ij}}
\le C \sumin d_i^2.  
\end{align}

Finally, we have, with $D_{ij}$ as in \eqref{dij},
\begin{align}
  \label{luc2}
[\QQQ,\SSS]_t=\sumin\sum_{j\sim i} d_i D_{ij}, 
\end{align}
and it follows similarly, using symmetry as in \eqref{ela},
\begin{align}\label{eld}
  \Var[\QQQ,\SSS]_t
&\le C \sum_{i\sim j,\, i'\sim j',\, \set{i,j}\cap\set{i',j'}\ge1}d_id_{i'}
\notag\\&
\le C 
\sum_{i,j,j':\,i\sim j,\, i\sim j'} (d_i+ d_j)(d_i+ d_{j'})
.\end{align}
It is easily seen that this sum can be estimated by
$C \sum_{|E(H)|\le 4}\hom(H,G)$,
summing over connected graphs $H$ with at most 4 edges.
Consequently,  
using again Sidorenko's inequality \cite{Sidorenko,SJ370} (or a more
complicated version of \eqref{salt}, which we leave to the reader)
\begin{align}\label{eld2}
  \Var[\QQQ,\SSS]_t &\le 
C \sum_{|E(H)|\le 4}\hom(H,G)
\le C \hom(\sK_{1,4},G)
%\notag\\&
=C\sumin d_i^4.
\end{align}
This completes the proof of \eqref{vQQ}--\eqref{vSN}.
\end{proof}

We will also use a version where we ``normalize'' $\SS(t)$ and $\SSS(t)$
by subtracting the average degree $\bd$ from the degrees in the definition.
We define, \cf{} \eqref{SSt} and \eqref{SSSt},
\begin{align}\label{RRt}
  \RR(t)&:=\sumin (d_i-\bd)\II_i(t)=\SS(t)-\bd\NN(t),
\\\label{RRRt}
\RRR(t)&:=\sumin (d_i-\bd)\III_i(t) =(1-t)\qw\RR(t)
=\SSS(t)-\bd\NNN(t).
\end{align}
Then $\RRR(t)$ is a martingale on $\oio$, with quadratic (co)variation,
similarly to \eqref{qSS}--\eqref{qQN},
\begin{align}
  \label{qRR}
[\RRR,\RRR]_t &= \sumin \indic{T_i\le t}(d_i-\bd)^2 (1-T_i)\qww,
\\\label{qQR}
  [\QQQ,\RRR]_t&=
\sumin \indic{T_i\le t} (1-T_i)\qww (d_i-\bd){\sum_{j\sim i}\III_j(T_i)},
\\\label{qRN}
  [\RRR,\NNN]_t&=
\sumin \indic{T_i\le t} (1-T_i)\qww (d_i-\bd)
.
\end{align}
Hence, or by \eqref{RRRt} and \eqref{eSS}--\eqref{eSN},
recalling \eqref{bd},
\begin{align}\label{eRR}
\E[\RRR,\RRR]_t &= \sumin (d_i-\bd)^2\cdot \frac{t}{1-t},
\\\label{eQR}
\E[\QQQ,\RRR]_t&=0,
\\\label{eRN}
\E[\RRR,\NNN]_t &= \sumin (d_i-\bd)\cdot \frac{t}{1-t}=0
.\end{align}

\begin{lemma}  \label{LU}
  For each fixed $t\in\oio$, we have, with constants $C$ that depend on $t$,
\begin{align}
\label{vRR}
\Var [\RRR,\RRR]_t &
\le C \sumin (d_i-\bd)^4,
\\\label{vQR}
\Var [\QQQ,\RRR]_t &
\le C \pmaxdx^2\sumin (d_i-\bd)^2,
\\\label{vRN}
\Var [\RRR,\NNN]_t &
\le C \sumin (d_i-\bd)^2
.  \end{align}
\end{lemma}
\begin{proof}
  First, \eqref{vRR} and \eqref{vRN} follow from \eqref{qRR} and \eqref{qRN}
since the terms in the sums are
  independent, similarly to \eqref{vSS} and \eqref{vNN} in \refL{LV}.

Next, by \eqref{RRRt}, \eqref{luc1}, and \eqref{luc2},
\begin{align}
  [\QQQ,\RRR]_t
=[\QQQ,\SSS]_t - \bd[\QQQ,\NNN]_t
=\sumin\sum_{j\sim i}(d_i-\bd)D_{ij},
\end{align}
with $D_{ij}$ defined in \eqref{dij}.
Recall that for a fixed $t$, 
the random variables $D_{ij}$
are uniformly bounded, and that
$D_{ij}$ and $D_{i'j'}$ are independent unless they have at least one common
index.
Hence, similarly to \eqref{eld}, it follows that
\begin{align}\label{luc3}
  \Var[\QQQ,\RRR]_t
&\le C \sum_{i\sim j,\, i'\sim j',\, \set{i,j}\cap\set{i',j'}\ge1}
 |d_i-\bd|\,|d_{i'}-\bd|
\notag\\&
\le C 
\sum_{i,j,j':\,i\sim j,\, i\sim j'} (|d_i-\bd|+ |d_j-\bd|)(|d_i-\bd|+ |d_{j'}-\bd|)
\notag\\&
=C\sum_{i}d_i^2|d_i-\bd|^2
+ C \sum_{i,j:i\sim j}d_i|d_i-\bd||d_j-\bd|
+\sum_i\Bigpar{\sum_{j\sim i}|d_j-\bd|}^2
.\end{align}
The first sum on the \rhs{} of \eqref{luc3} is clearly at most 
$\pmaxdx^2\sumin (d_i-\bd)^2$.
The second sum is
\begin{align}
 & \le \maxdx \sum_{i,j:i\sim j}\bigpar{|d_i-\bd|^2+|d_j-\bd|^2}
\notag\\&
=\maxdx\sum_id_i|d_i-\bd|^2+\maxdx\sum_jd_j|d_j-\bd|^2
\notag\\&
\le 2\pmaxdx^2\sum_i|d_i-\bd|^2.
\end{align}
Finally, the third sum is, by the \CSineq,
\begin{align}
 & \le \sum_id_i\sum_{j\sim i}|d_j-\bd|^2
\le \pmaxdx\sum_jd_j|d_j-\bd|^2
%\notag\\&
\le \pmaxdx^2\sum_j|d_j-\bd|^2.
\end{align}
Combining these estimates, we obtain \eqref{vQR}.
\end{proof}

\section{Proofs of convergence and main theorems}
\label{Spf2}
%\section{Martingale convergence}\label{Spf2}

We are now prepared to show convergence of the martingales defined in
\refS{Spf1}, and then to prove the theorems in \refS{Smain}.
Although \refTs{TA} and \ref{TB} can be proved as special cases of
\refT{TC}, we have chosen to give separate proofs of the three theorems,
using the same general method but with some variations;
this illustrates the differences between the cases.
We begin with the sparse case in \refT{TA}, which shows the main ideas
without unnecessary complications.
We first state a lemma, where we assume that the graphs $G\nn$ are
non-random.
(Thus the assumptions \eqref{ta1}--\eqref{ta3} are replaced by their
non-random counterparts \eqref{la1}--\eqref{la3}.)

In this section we use convergence in both $D\oio$ and $D\oi$.
For processes defined on $\oi$, we will always use convergence in $D\oi$,
even when this is not explicitly said.

\begin{lemma}\label{LA}
Assume that $G\nn$ is a sequence of non-random graphs with $V(G\nn)=[n]$
such that, 
as in \refT{TA},
\begin{gather}
\frac{1}{n}\sumin d\nn_i=\frac{2|E\nn|}{n}\to\dx\in\ooo,\label{la1}
\\
\frac{1}{n}\sumin (d\nn_i)^2\to\chix\in\ooo,\label{la2}
\\
\maxdx\nn:=\max_i d\nn_i = o(n\qq).\label{la3}
\end{gather}
\begin{romenumerate}
  
\item \label{LA1}
  Then, in $D\oio$,
  \begin{align}\label{lao}
    n\qqw\bigpar{\QQQ\nn(t),\SSS\nn(t),\NNN\nn(t)}
\dto \ZZZ(t)=\bigpar{\ZZZ_\sQ(t),\ZZZ_\sS(t),\ZZZ_\sN(t)},
  \end{align}
where $\ZZZ(t)$ is a continuous Gaussian process on $\oio$ with $\E
\ZZZ(t)=0$
and covariance function,
for $ 0\le s\le t<1$,
\begin{align}\label{la4}
 \Cov\bigpar{\ZZZ(s),\ZZZ(t)}
=
 \E\bigpar{\ZZZ(s)\ZZZ(t)\tr}
=
\gSSS(s):=
  \begin{pmatrix}
\frac{\dx}{2}\frac{s^2}{(1-s)^2} & 0 & 0 \\
0 & \chix\frac{s}{1-s}  & \dx \frac{s}{1-s}\\
0 & \dx\frac{s}{1-s} &  \frac{s}{1-s}  %\\
  \end{pmatrix}
%\qquad 0\le s\le t<1.
.\end{align}

\item\label{LA2}
Similarly, in $D\oi$,
  \begin{align}\label{la}
    n\qqw\bigpar{\QQ\nn(t),\SS\nn(t),\NN\nn(t)}
\dto \ZZ(t)=\bigpar{\ZZ_\sQ(t),\ZZ_\sS(t),\ZZ_\sN(t)},
  \end{align}
where $\ZZ(t)$ is a continuous Gaussian process on $\oi$ with $\E\ZZ(t)=0$
and covariance function,
for $ 0\le s\le t\le1$,
\begin{align}\label{laa}
 \Cov\bigpar{\ZZ(s),\ZZ(t)}
=
\gSS(s,t):=
  \begin{pmatrix}
\frac{\dx}{2}{s^2}{(1-t)^2} & 0 & 0 \\
0 & \chix s(1-t)  & \dx s(1-t)\\
0 & \dx s(1-t) &  s(1-t)  %\\
  \end{pmatrix}
.\end{align}

\end{romenumerate}
\end{lemma}
\begin{proof}
\pfitemref{LA1}
Consider the vector-valued
martingale on $\oio$ defined by
\begin{align}\label{maa}
M\nn(t)=\bigpar{M\nn_\sQ(t),M\nn_\sS(t),M\nn_\sN(t)}
:=n\qqw\bigpar{\QQQ\nn(t),\SSS\nn(t),\NNN\nn(t)}.
  \end{align}
For any fixed $t\in\oio$, \eqref{eQQ}--\eqref{eSN} 
together with \eqref{la1}--\eqref{la2} and \eqref{la4}
show that
the matrix of quadratic covariations
has expectation
\begin{align}\label{mab}
  \E[M\nn,(M\nn)\tr]_t \to \gSSS(t),
\end{align}
while \refL{LV} shows that for each $i,j\in\set{\sQ,\sS,\sN}$ we have, for
the corresponding $\XXXX,\YYY\in\set{\QQQ,\SSS,\NNN}$,
using \eqref{la2}--\eqref{la3},
\begin{align}\label{mac}
  \Var\bigsqpar{M\nn_i,M\nn_j}_t&
= \Var\bigpar{n\qw[\XXXX\nn,\YYY\nn]_t}
\le C n\qww\Bigpar{\sumin \bigpar{d_i\nn}^4+n}
\notag\\&
\le C n\qww \xpar{\maxdx\nn}^2\sumin \bigpar{d_i\nn}^2+Cn\qw
\to0.
\end{align}
Hence, \refP{P:JS} applies and yields the result.

\pfitemref{LA2}
Convergence in $D\oio$ follows immediately from \ref{LA1} and the relations
\eqref{QQQt}--\eqref{NNNt}, defining 
$\ZZ_\sQ(t):=(1-t)^2\ZZZ_\sQ(t)$,
$\ZZ_\sS(t):=(1-t)\ZZZ_\sS(t)$, and
$\ZZ_\sN(t):=(1-t)\ZZZ_\sN(t)$
for $t\in\oio$. We define further $\ZZ_\sQ(1):=\ZZ_\sS(1):=\ZZ_\sN(1):=0$;
continuity of $\ZZ(t)$ at $t=1$ and convergence in $D\oi$ then follows by
\refL{Loi} and \refR{Roi}, taking $a_\sQ=b_\sQ=2$ and $a_\sS=b_\sS=a_\sN=b_\sN=1$
and recalling \eqref{qm3} and \eqref{eQQ}--\eqref{eNN}.
\end{proof}

\begin{remark}\label{RBB}
  In particular, $\ZZ_\sN(t)$ has covariance function $s(1-t)$ for $0\le
  s\le t\le 1$, and is thus a Brownian bridge $\BR(t)$, see \eqref{brb}.
The convergence of $n\qqw\NN\nn(t)=n\qqw(N\nn(t)-nt)$ 
to a Brownian bridge is a well-known fact,
since by \eqref{nt}, $n\qw N\nn(t)$ is the empirical distribution function of
the \iid{} uniformly distributed random variables $T_i$, $i\in[n]$;
see \eg{} \cite[Theorems 16.4 and 13.1]{Billingsley}.
This holds for any graphs $G\nn$ (without any conditions on the degrees), 
since the edges do not affect $N\nn(t)$, and therefore this result for
$\NN\nn(t)$
will return in other
proofs below.
\end{remark}

%\section{Proofs of main theorems}\label{Spf}

\begin{proof}[Proof of \refT{TA}]
Assume first that the graphs $G\nn$ are non-random.
In particular, the variables in \eqref{ta1}--\eqref{ta3} are non-random, and 
the limits there are thus usual limits of real numbers.
Hence, \refL{LA} applies. We  prove first the \ctime{} result \ref{TAc}, and
then use it to derive the \dtime{} result \ref{TAd}.

\pfitemx{\emph{$G\nn$ non-random}, \ref{TAc}}
By  \eqref{b2a} and \refL{LA}\ref{LA2},
using also \refL{Ladd},
\begin{align}\label{tb1}
  n\qqw\bigpar{L\nn(t)-t^2|E\nn|}
=
  n\qqw\bigpar{\QQ\nn(t)+t\SS\nn(t)}
\dto 
Z(t):=
\ZZ_\sQ(t)+t\ZZ_\sS(t)
\end{align}
in $D\oi$.
The process $Z(t)$ is clearly continuous and Gaussian, with mean $\E Z(t)=0$
and covariance function, using \eqref{la}--\eqref{laa},
%for $0\le s\le t\le 1$,
\begin{align}
  \Cov\bigpar{Z(s),Z(t)}&
=  \Cov\bigpar{\ZZ_\sQ(s),\ZZ_\sQ(t)}
+st \Cov\bigpar{\ZZ_\sS(s),\ZZ_\sS(t)}
\notag\\&
=\frac{\dx}2s^2(1-t)^2+\chix s^2t(1-t).
\end{align}
This proves \eqref{tac}--\eqref{tac2}.

\pfitemx{\emph{$G\nn$ non-random}, \ref{TAd}}
The proof just given for \ref{TAc} shows that the result \eqref{tb1} holds
jointly with $n\qqw \NN(t)\dto \ZZ_\sN(t)$.
Hence, we can apply \refT{TN}, with $\iX\nn_k:=\iL_k(t)$, $X\nn(t):=L\nn(t)$,
$W(t)=\ZZ_\sN(t)$,
$a_n=n\qqw$, $b_n=2|E\nn|$, $f(t)=t^2/2$, and thus $c=\lim_\ntoo 2|E\nn|/n=\dx$.
This yields \eqref{tad} with $\iZ(t):=Z(t)-\dx t \ZZ_\sN(t)$,
and \eqref{tnc} yields the covariance function \eqref{tad2}.

\emph{$G\nn$ random}:
Finally, consider the general case when $G\nn$ may be random.
By the Skorohod coupling theorem \cite[Theorem~4.30]{Kallenberg},
we may for convenience
assume that the limits in \eqref{ta1}--\eqref{ta3} hold a.s.
We then condition on the sequence $(G\nn)\xoo$, and note that thus
\as{} the deterministic case just proved applies to the sequence.
Hence, \as{} the conclusions \eqref{tad} and \eqref{tac} hold conditionally
on $(G\nn)\xoo$. Since the limits have distributions determined by
\eqref{tad2} and \eqref{tac2} which do not depend on the 
sequence $(G\nn)\xoo$, it follows that \eqref{tad} and \eqref{tac} hold
unconditionally too.
\end{proof}

We turn to the regular case, again beginning with a lemma for non-random
 $G\nn$.

\begin{lemma}\label{LB}
Assume that $G\nn$ is a sequence of non-random graphs with $V(G\nn)=[n]$
such that, 
as in \refT{TB},
$G\nn$ is regular with degree $1\le d\nn=o(n)$.
Then $\SS\nn(t)=d\nn\NN\nn(t)$
and $\SSS\nn(t)=d\nn\NNN\nn(t)$. Furthermore:
\begin{romenumerate}
  
\item \label{LB1}
We have,  in $D\oio$,
  \begin{align}\label{lbo}
   \bigpar{(nd\nn)\qqw\QQQ\nn(t),n\qqw\NNN\nn(t)}
\dto \ZZZ(t)=\bigpar{\ZZZ_\sQ(t),\ZZZ_\sN(t)},
  \end{align}
where $\ZZZ_\sQ(t)$ and $\ZZZ_\sN(t)$ are independent
continuous Gaussian processs on
$\oio$ with means $0$
%$\ZZZ(t)=0$
and covariance functions,
for $ 0\le s\le t<1$,
\begin{align}
 \Cov\Bigpar{\ZZZ_\sQ(s),\ZZZ_\sQ(t)}&
=
%\gSSS(s):=
%  \begin{pmatrix}
\frac{1}{2}\frac{s^2}{(1-s)^2},
\\
 \Cov\Bigpar{\ZZZ_\sN(s),\ZZZ_\sN(t)}&
=
 \frac{s}{1-s}
%\qquad 0\le s\le t<1.
.\end{align}

\item\label{LB2}
Similarly, in $D\oi$,
  \begin{align}\label{lb}
   \bigpar{(nd\nn)\qqw\QQ\nn(t),n\qqw\NN\nn(t)}
\dto \ZZ(t)=\bigpar{\ZZ_\sQ(t),\ZZ_\sN(t)},
  \end{align}
where $\ZZ_\sQ(t)$ and $\ZZ_\sN(t)$ are 
independent continuous Gaussian processs on
$\oio$ with means $0$
%$\ZZ(t)=0$
and covariance functions,
for $ 0\le s\le t\le1$,
\begin{align}\label{lb2}
 \Cov\Bigpar{\ZZ_\sQ(s),\ZZ_\sQ(t)}&
=
\frac{1}{2}{s^2}{(1-t)^2},
\\\label{lb2b}
 \Cov\Bigpar{\ZZ_\sN(s),\ZZ_\sN(t)}&
=
s(1-t)
.\end{align}

\end{romenumerate}
\end{lemma}

\begin{proof}
Since $G\nn$ is regular with $d\nn_i=d\nn$ for every $i\in[n]$, it follows from 
\eqref{SSt}--\eqref{NNt}  and \eqref{SSSt}--\eqref{NNNt} 
that $\SS\nn(t)=d\nn\NN\nn(t)$ and $\SSS\nn(t)=d\nn\NNN\nn(t)$.

The rest of the proof is similar to the proof of \refL{LA}.
We ignore $\SSS\nn$ and define now
\begin{align}\label{mba}
M\nn(t)=\bigpar{M\nn_\sQ(t),M\nn_\sN(t)}
:=\bigpar{(nd\nn)\qqw\QQQ\nn(t),n\qqw\NNN\nn(t)}.
  \end{align}
For any fixed $t\in\oio$, \eqref{eQQ}, \eqref{eNN} and \eqref{eQN} 
together with $|E\nn|=nd\nn/2$
show that
the matrix of quadratic covariations
has expectation
\begin{align}\label{mbb}
  \E[M\nn,(M\nn)\tr]_t 
=
\gS(t):=
 \begin{pmatrix}
\frac{1}{2}\frac{t^2}{(1-t)^2} & 0
\\
0 & \frac{t}{1-t}
\end{pmatrix}
\end{align}
(in this case an identity for all $n$),
while \refL{LV} shows that 
\begin{align}
  \Var [M\nn_\sQ,M\nn_\sQ]_t& 
= \Var\bigpar{(nd\nn)\qw[\QQQ\nn,\QQQ\nn]_t}
\le (nd\nn)\qww C n(d\nn)^3
\notag\\&
= C d\nn/n = o(1),
\\
  \Var [M\nn_\sQ,M\nn_\sN]_t&
 = \Var\bigpar{(n^2d\nn)\qqw[\QQQ\nn,\NNN\nn]_t}
\le (n^2d\nn)\qw C n(d\nn)^2
\notag\\&
= C d\nn/n = o(1),
\\
  \Var [M\nn_\sN,M\nn_\sN]_t& 
= \Var\bigpar{n\qw[\NNN\nn,\NNN\nn]_t}
\le n\qww C n
= C n\qw = o(1),
\end{align}
Hence, \refP{P:JS} applies and yields \ref{LB1}, with $\Cov(\ZZZ(t))=\gS(t)$
in \eqref{mbb}.
The independence of $\ZZZ_\sQ$ and $\ZZZ_\sN$ follows from the fact that 
the  matrix $\gS(t)$  is diagonal, and thus
all covariances 
$\Cov\bigpar{\ZZZ_\sQ(s),\ZZZ_\sN(t)}=\Var\ZZZ_{\sQ}(s)$ ($0\le s\le t$)
vanish.

Finally, \ref{LB2} follows as in the proof of \refL{LA} by \refL{Loi} and
\refR{Roi}, using \eqref{eQQ} and \eqref{eNN}.
\end{proof}

\begin{proof}[Proof of \refT{TB}]
  As in the proof of \refT{TA}, we may by conditioning assume that each
  $G\nn$ is non-random. 
Then \refL{LB} applies.
As in the proof of \refT{TA}, we first consider continuous time, but this
time we cannot derive \ref{TB1} from \ref{TB2}, since the normalizing
factors are different; hence we derive instead
first the intermediate \ctime{} result \eqref{lb4} below.

\pfitemref{TB1}
The decomposition \eqref{b2a} yields, noting that 
$\SS\nn(t)=d\nn\NN\nn(t)$ by \refL{LB}
and $|E\nn|=nd\nn/2$ by \eqref{di},
and using also \eqref{NNt},
\begin{align}\label{nov3}
  L\nn(t)-\frac{d\nn}{2n}N\nn(t)^2&
=\QQ\nn(t)+td\nn\NN\nn(t)+t^2\frac{nd\nn}{2}
 -\frac{d\nn}{2n}\bigpar{\NN\nn(t)+nt}^2 
\notag\\&
=\QQ\nn(t)-\frac{d\nn}{2n}{\NN\nn(t)}^2.
\end{align}
As a consequence of \eqref{lb}, we have
$\sup_t|\NN\nn(t)|=\Op\bigpar{n\qq}$;
recall that we write this as
$\NN\nn(t)=\Opx\bigpar{n\qq}$.
Hence,  \eqref{nov3} implies, using $d\nn=o(n)$,
\eqref{lb}, and \refL{L0},
\begin{align}\label{lb4}
(nd\nn)\qqw\Bigpar{  L\nn(t)-\frac{d\nn}{2n}N\nn(t)^2}&
=(nd\nn)\qqw\QQ\nn(t)+\Opx((d\nn)\qq n\qqw)
\notag\\&
=(nd\nn)\qqw\QQ\nn(t)+\opx(1)
\notag\\&
\dto \ZZ_\sQ(t).
\end{align}
Furthermore, \eqref{lb} implies that \eqref{lb4} 
holds jointly with $n\qqw(N\nn(t)-nt)\dto \ZZ_\sN(t)$.
 Thus by \refT{TN} (with $b_n=f(t)=c=0$), 
or  in this simple case directly by substituting $\tau_\fnt$ for $t$
in \eqref{lb4},
\begin{align}\label{lb5}
(nd\nn)\qqw\Bigpar{  \iL_\fnt-\frac{d\nn}{2n}\fnt^2}&
\dto \ZZ_\sQ(t).
\end{align}
Furthermore, we have
\begin{align}
  \frac{d\nn}{2n}\fnt^2-t^2|E\nn|
=
  \frac{d\nn}{2n}\bigpar{\fnt^2-(nt)^2}
=O(d\nn)=o\bigpar{(nd\nn)\qq},
\end{align}
and thus \eqref{lb5} implies \eqref{tbd} with 
$\iZ(t)=\ZZ_\sQ(t)$.
Hence, \eqref{tbd2} holds by \eqref{lb2}.

\pfitemx{\ref{TB2} and \ref{TB3}}
It follows from \eqref{bd}, \eqref{NNt} and \eqref{lb} that
\begin{align}\label{lb6}
\frac{1}{n\qq d\nn}\Bigpar{\frac{d\nn}{2n}N\nn(t)^2-t^2|E\nn|}
&=
 \frac1{2 n^{3/2}}\bigpar{N\nn(t)^2-n^2t^2}
\notag\\&
=
\frac1{2 n^{3/2}}\bigpar{\NN\nn(t)^2+2nt\NN\nn(t)}
\notag\\&
\dto t \ZZ_\sN(t).
\end{align}
Furthermore, this holds jointly with \eqref{lb4}
and its consequence
\begin{align}\label{lb7}
(n\qq d\nn)\qw\Bigpar{  L\nn(t)-\frac{d\nn}{2n}N\nn(t)^2}&
\dto \doo\qqw\ZZ_\sQ(t).
\end{align}
Combining \eqref{lb7} and \eqref{lb6} yields
\begin{align}\label{lb8}
&\frac{1}{n\qq d\nn}\bigpar{L\nn(t)-t^2|E\nn|}
\dto Z(t):=\doo\qqw\ZZ_\sQ(t) + t \ZZ_\sN(t).
\end{align}
This proves \eqref{tbdii}, and \eqref{lb2}--\eqref{lb2b} imply that the
covariance 
function is given by \eqref{tbdii2}; when $\doo=\infty$ this simplifies to 
\eqref{tbdii3}, and we see also directly from \eqref{lb8} that in this case
$Z(t)=t\ZZ_\sN(t)=t\BR(t)$, see \refR{RBB}.
\end{proof}

Finally, we treat the general case in \refT{TC}, again beginning with a lemma.

\begin{lemma}\label{LC}
  Assume that $G\nn$ is a sequence of non-random graphs with $V(G\nn)=[n]$
and that $\gb_n$ is a sequence of positive constants
such that, 
as in \refT{TC},
for some  constants $\gl_1,\gl_2\in\ooo$, 
we have, as \ntoo,
\begin{gather}
\gb_n=o(n),\label{lc0}
\\
\frac{2|E\nn|}{\gb_n^2}=\frac{n\bd\nn}{\gb_n^2}\to\gl_1,\label{lc1}
\\
\frac{1}{\gb_n^2}\sumin \bigpar{d\nn_i-\bd\nn}^2\to\gl_2,\label{lc2}
\\
\maxdx\nn =o(\gb_n).\label{lc3}
\end{gather}
\begin{romenumerate}  
\item \label{LC1}
  Then, in $D\oio$,
  \begin{align}\label{lco}
    \bigpar{\gb_n\qw\QQQ\nn(t),\gb_n\qw\RRR\nn(t),n\qqw\NNN\nn(t)}
\dto \ZZZ(t)=\bigpar{\ZZZ_\sQ(t),\ZZZ_\sR(t),\ZZZ_\sN(t)},
  \end{align}
where $\ZZZ_\sQ(t)$, $\ZZZ_\sR(t)$, and $\ZZZ_\sN(t)$ are independent
continuous
Gaussian processes on $\oio$ with means $0$
and covariance functions, 
for $ 0\le s\le t< 1$,
\begin{align}
 \Cov\Bigpar{\ZZZ_\sQ(s),\ZZZ_\sQ(t)}&
=
\frac{\gl_1}{2}\frac{s^2}{(1-s)^2},
\\
 \Cov\Bigpar{\ZZZ_\sR(s),\ZZZ_\sR(t)}&
=
\gl_2\frac{s}{1-s}, 
\\
 \Cov\Bigpar{\ZZZ_\sN(s),\ZZZ_\sN(t)}&
=
 \frac{s}{1-s}
%\qquad 0\le s\le t<1.
.\end{align}

\item\label{LC2}
Similarly, in $D\oi$,
  \begin{align}\label{lc}
    \bigpar{\gb_n\qw\QQ\nn(t),\gb_n\qw\RR\nn(t),n\qqw\NN\nn(t)}
\dto \ZZ(t)=\bigpar{\ZZ_\sQ(t),\ZZ_\sR(t),\ZZ_\sN(t)},
  \end{align}
where $\ZZ_\sQ(t)$, $\ZZ_\sR(t)$, and $\ZZ_\sN(t)$ are independent
continuous
Gaussian processes on $\oi$ with means $0$
and covariance functions, 
for $ 0\le s\le t\le1$,
\begin{align}\label{lcQ}
 \Cov\Bigpar{\ZZ_\sQ(s),\ZZ_\sQ(t)}&
=
\frac{\gl_1}{2}{s^2}{(1-t)^2},
\\\label{lcR}
 \Cov\Bigpar{\ZZ_\sR(s),\ZZ_\sR(t)}&
=
\gl_2{s}{(1-t)}, 
\\\label{lcN}
 \Cov\Bigpar{\ZZ_\sN(s),\ZZ_\sN(t)}&
=s(1-t)
%\qquad 0\le s\le t<1.
.\end{align}
\end{romenumerate}
\end{lemma}

\begin{proof}
  This is similar to the proofs of \refLs{LA} and \ref{LB}.
We now define the martingale
\begin{align}\label{mca}
M\nn(t)=\bigpar{M\nn_\sQ(t),M\nn_\sR(t),M\nn_\sN(t)}
:=\bigpar{\gb_n\qw\QQQ\nn(t),\gb_n\qw\RRR\nn(t),n\qqw\NNN\nn(t)}.
  \end{align}
For any fixed $t\in\oio$, \eqref{eQQ}, \eqref{eNN}, \eqref{eQN} and
\eqref{eRR}--\eqref{eRN}  
together with \eqref{lc1}--\eqref{lc2}
show that
the matrix of quadratic covariations
has expectation
\begin{align}\label{mcb}
  \E[M\nn,(M\nn)\tr]_t \to \gSSS(t)
:=
\begin{pmatrix}
\frac{\gl_1}{2}\frac{t^2}{(1-t)^2} & 0 & 0 \\
0 & \gl_2\frac{t}{1-t}  & 0\\
0 & 0 &  \frac{t}{1-t}  %\\
  \end{pmatrix}
.\end{align}
Similarly, \refLs{LV} and \ref{LU} 
together with \eqref{lc0}--\eqref{lc3} % and $\gb_n=o(n)$
show that 
(omitting superscripts ${}\nn$ for convenience),
for any fixed $t\in\oio$,
\begin{align}
  \Var[M\nnq_\sQ,M\nnq_\sQ]_t 
&=\gb_n^{-4}\Var[\QQQ\nnq,\QQQ\nnq]_t
\le C \gb_n^{-4}\sumin d_i\nnq^3
\notag\\&
\le C \gb_n^{-4}\gD\nnq^2\sumin d_i
= C \Bigpar{\frac{\gD\nnq}{\gb_n}}^2\cdot\frac{2|E\nnq|}{\gb_n^2}
\to0,
\\
  \Var[M\nnq_\sR,M\nnq_\sR]_t 
&=\gb_n^{-4}\Var[\RRR\nnq,\RRR\nnq]_t
\le C \gb_n^{-4}\sumin (d_i\nnq-\bd\nnq)^4
\notag\\&
\le C \gb_n^{-4}\gD\nnq^2\sumin (d\nnq_i-\bd\nnq)^2
=C \Bigpar{\frac{\gD\nnq}{\gb_n}}^2
 \cdot\frac{1}{\gb_n^2}\sumin (d\nnq_i-\bd\nnq)^2
%\le C \gb_n^{-2}\gD\nnq^2
\to0,
\\
  \Var[M\nnq_\sN,M\nnq_\sN]_t 
&=n^{-2}\Var[\NNN\nnq,\NNN\nnq]_t\le Cn\qw\to0,
\\
  \Var[M\nnq_\sQ,M\nnq_\sR]_t 
&=\gb_n^{-4}\Var[\QQQ\nnq,\RRR\nnq]_t
\le C \gb_n^{-4}\gD\nnq^2\sumin (d\nnq_i-\bd\nnq)^2
%\notag\\&
%\le C \gb_n^{-2}\gD\nnq^2
\to0,
\\
  \Var[M\nnq_\sQ,M\nnq_\sN]_t 
&=\gb_n^{-2}n\qw\Var[\QQQ\nnq,\NNN\nnq]_t
\le C \gb_n^{-2}\frac{1}{n}\sumin d\nnq_i^2
\notag\\&
= C \gb_n^{-2}\Bigpar{\frac{1}{n}\sumin (d\nnq_i-\bd\nnq)^2+\bd\nnq^2}
\le \frac C n + C \lrpar{\frac{\gb_n}{n}\cdot\frac{n\bd}{\gb_n^2}}^2
\to0,
\\
  \Var[M\nnq_\sR,M\nnq_\sN]_t 
&=\gb_n^{-2}n\qw\Var[\RRR\nnq,\NNN\nnq]_t
\le C n\qw\gb_n^{-2}\sumin (d\nnq_i-\bd\nnq)^2
\le C n\qw 
\to0
.\end{align}
Hence, \refP{P:JS} applies and yields the result; the three components of
the limit process are independent since the matrix $ \gSSS(t)$ in
\eqref{mcb} is diagonal.

Finally, \ref{LC2} follows as in the proof of \refL{LA} by \refL{Loi} and
\refR{Roi}, using \eqref{eQQ}, \eqref{eNN}, and \eqref{eRR}
together with \eqref{lc1}--\eqref{lc2}.
%$\ZZ_\sQ(t):=(1-t)^2\ZZZ_\sQ(t)$,
%$\ZZ_\sR(t):=(1-t)\ZZZ_\sR(t)$, and
%$\ZZ_\sN(t):=(1-t)\ZZZ_\sN(t)$
\end{proof}

\begin{proof}[Proof of \refT{TC}]
As in the proof of \refTs{TA} and \ref{TB}, we may by conditioning assume
that $G\nn$ are non-random. Then \refL{LC} applies.  
As in the proof of \refT{TB}, we first consider continuous time and derive
an intermediate result.

\pfitemref{TCd}
We now use \eqref{RRt} and \eqref{di}
to write the decomposition \eqref{b2a}
as
\begin{align}\label{luc6}
  L\nn(t)
=\QQ\nn(t)+t\RR\nn(t)+t\bd\nn\NN\nn(t)+t^2\frac{n\bd\nn}{2}
.\end{align}
Hence,
using also \eqref{NNt},
\cf{} the regular case \eqref{nov3} where $\RR\nn(t)=0$,
\begin{align}\label{luc7}
  L\nn(t)-\frac{\bd\nn}{2n}N\nn(t)^2&
=L\nn(t) -\frac{\bd\nn}{2n}\bigpar{\NN\nn(t)+nt}^2 
\notag\\&
=\QQ\nn(t)+t\RR\nn(t)-\frac{\bd\nn}{2n}{\NN\nn(t)}^2.
\end{align}
As a consequence of \eqref{lc}, we have
$\NN\nn(t)=\Opx\bigpar{n\qq}$.
Hence,  \eqref{luc7} implies, using $\bd\nn\le\gD\nn=o(\gb_n)$,
\eqref{lc}, and \refL{L0},
\begin{align}\label{lc4}
\gb_n\qw\Bigpar{  L\nn(t)-\frac{\bd\nn}{2n}N\nn(t)^2}&
=\gb_n\qw\QQ\nn(t)+\gb_n\qw t\RR\nn(t)+\Opx(\bd\nn/\gb_n)
\notag\\&
=\gb_n\qw\QQ\nn(t)+t\gb_n\qw\RR\nn(t)+\opx(1)
\notag\\&
\dto \ZZ_\sQ(t)+t\ZZ_\sR(t).
\end{align}
Furthermore, \eqref{lc} implies that \eqref{lc4} 
holds jointly with $n\qqw(N\nn(t)-nt)\dto \ZZ_\sN(t)$.
Thus by \refT{TN} (with $b_n=f(t)=c=0$), 
or directly by substituting $\tau_\fnt$ for $t$
in \eqref{lc4},
\begin{align}\label{lc5}
\gb_n\qw\Bigpar{  \iL_\fnt-\frac{\bd\nn}{2n}\fnt^2}&
\dto \iZ(t):=\ZZ_\sQ(t)+t\ZZ_\sR(t).
\end{align}
Furthermore, we have
\begin{align}
  \frac{\bd\nn}{2n}\fnt^2-t^2|E\nn|
=
  \frac{\bd\nn}{2n}\bigpar{\fnt^2-(nt)^2}
=O(\bd\nn)=O(\gD\nn)=o\bigpar{\gb_n},
\end{align}
and thus \eqref{lc5} implies \eqref{tcd}.
We have \eqref{tcd2} by \eqref{lcQ}--\eqref{lcR}.

\pfitemref{TCc}
By \eqref{bd}, \eqref{NNt} and \eqref{lc}, 
we have,
as in \eqref{lb6}, 
\begin{align}\label{lc6}
\frac{1}{n\qq\bd\nn}\Bigpar{\frac{\bd\nn}{2n}N\nn(t)^2-t^2|E\nn|}
&=
% \frac12 n^{-3/2}\bigpar{N\nn(t)^2-n^2t^2}
%\notag\\&
%=
\frac1{2 n^{3/2}}\bigpar{\NN\nn(t)^2+2nt\NN\nn(t)}
\notag\\&
\dto t \ZZ_\sN(t).
\end{align}
(The case $\bd\nn=0$ is trivial and may be excluded.)
Furthermore, this holds jointly with \eqref{lc4}.

\pfitemref{TCca}
If $\ga<\infty$, then \eqref{tc4} and \eqref{lc6} imply
\begin{align}\label{lc6x}
\frac{1}{\gb_n}\Bigpar{\frac{\bd\nn}{2n}N\nn(t)^2-t^2|E\nn|}
%\notag\\&
\dto \ga t \ZZ_\sN(t),
\end{align}
jointly with \eqref{lc4}.
Consequently, recalling also \eqref{bd}, 
\eqref{tcc} holds with 
\begin{align}\label{lc7}
Z(t):=
\iZ(t)+\ga t \ZZ_\sN
=\ZZ_\sQ(t)+t\ZZ_\sR(t) +\ga t \ZZ_\sN.
\end{align}
The covariance function \eqref{tcc2} follows from \eqref{lcQ}--\eqref{lcN}.

\pfitemref{TCcb}
If $\ga=\infty$, then $\gb_n=\op\xpar{n\qq\bd\nn}$. 
 In this case, $L\nn(t)$ is dominated by the contribution from
\eqref{lc6}.
More precisely, \eqref{lc4} now implies
\begin{align}\label{lc4b}
\frac{1}{n\qq\bd\nn}\Bigpar{  L\nn(t)-\frac{\bd\nn}{2n}N\nn(t)^2}&
\dto0,
\end{align}
which together with \eqref{lc6} yields \eqref{tcdii} with $Z(t):=t\ZZ_\sN(t)$.
By \refR{RBB}, $\ZZ_\sN(t)$ is a Brownian bridge $\BR(t)$,
and the result follows.
\end{proof}

\section{Number of components}\label{Scomp}

Consider the case when $G$ is a tree; 
in this case, 
the visible part of the graph is a forest. Let $\iK\nn_k$ and $K\nn(t)$ 
be the number of components in the visible forest at time $k$ or $t$,
repectively, for the \dtime{} and \ctime{} versions.
\begin{theorem}\label{TK}
  Assume that $G\nn$ is a sequence of deterministic or random trees with
  $V(G\nn)=[n]$, and that \eqref{ta2} and \eqref{ta3} hold.
Let $\gamx:=\chix-4$.  
For the number of components in the visible forest, we then have, in $D\oi$,
\begin{align}\label{ktad}
  n\qqw\bigpar{\iK\nn_\fnt-t(1-t)n}
\dto
\iZ(t),
\end{align}
 and
\begin{align}\label{ktac}
  n\qqw\bigpar{K\nn(t)-t(1-t)n}
\dto
\bcZ(t),
\end{align}
where $\iZ(t)$ and $\bcZ(t)$ are continuous Gaussian processes with
$\E\iZ(t)=\E\bcZ(t)=0$ and covariance functions given by, 
for $ 0\le s\le t\le1$,
\begin{align}\label{ktad2}
 \Cov\bigpar{\iZ(s),\iZ(t)}
=
%\igs(s,t):=
{s^2}{(1-t)^2} + \gamx s^2t(1-t)
,\end{align}
as in \eqref{tad2},
and
\begin{align}\label{ktac2}
 \Cov\bigpar{\bcZ(s),\bcZ(t)}
=
%\bcgs(s,t):=
{s^2}{(1-t)^2} + \gamx s^2t(1-t)+s(1-2s)(1-t)(1-2t)
.\end{align}
\end{theorem}

\begin{proof}
  In the \dtime{} version, at time $k$ the visible forest has
$k$ vertices and $\iL\nn_k$ edges, and thus 
\begin{align}\label{kl0}
\iK\nn_k=k-\iL\nn_k.
\end{align}
Hence,
\begin{align}\label{kl1}
\iK\nn_\fnt-t(1-t)n
%=t^2n-\iL\nn_\fnt + O(1)
=-\bigpar{\iL\nn_\fnt-t^2 n} + O(1)
.\end{align}
\refC{CA} applies and thus \eqref{tad}--\eqref{tac2} hold, with $\dx=2$ and
thus $\gamx=\chix-4$.
Hence, \eqref{ktad} follows from \eqref{kl1} and \eqref{tad}, together with
the fact that 
$-\iZ(t)\eqd \iZ(t)$ (as processes); the covariances \eqref{ktad2} 
are given in \eqref{tad2}.

Similarly, recalling \eqref{NNt},
\begin{align}\label{kl2}
  K\nn(t)&=N\nn(t)-L\nn(t)
\notag\\&
=nt(1-t)+\NN\nn(t)-\bigpar{L\nn(t)-t^2|E\nn|}+O(1).
\end{align}
Hence, \eqref{ktac} follows from \eqref{la} and \eqref{tb1}
(which  hold jointly), with 
\begin{align}\label{rab}
\bcZ(t):=\ZZ_\sN(t)-\bigpar{\ZZ_\sQ(t)+t\ZZ_\sS(t)}
=-\ZZ_\sQ(t)-t\ZZ_\sS(t)+\ZZ_\sN(t),
\end{align}
and the covariances \eqref{ktac2} follow from \eqref{rab} and \eqref{laa}.
Alternatively, \eqref{ktac} and \eqref{ktac2} follow from  \refT{TN2} 
with $a_n=n\qqw$, $b_n=n$, $c=1$ and $f(t)=t(1-t)$.
\end{proof}

\begin{remark}
  The case $G=\sP_n$ was studied in \cite{SJ198}, where the results in
\refT{TK} where proved for this case (using the same method as here),
which solved a problem from \cite{afH}.
The main results in \cite{SJ198} concern asymptotics of the maximum
$\max_k \iK_k\nn$ and the difference
$\max_k\iK_k\nn-\iK_{\ceil{n/2}}\nn$.
(Note that the maximum is attained for $k$ close to $n/2$.)
The proofs of these results in \cite{SJ198} are easily modified to the
present more general case, using \eqref{ktad} or \eqref{ktac};
we leave the details to the reader.
\end{remark}

\refT{TK} extends easily to forests $G\nn$; again the visible part is always
a forest.
It would be interesting to have similar results for general graphs $G$, but
this seems to require different methods.

\begin{problem}
Study the number of components in the visible part of $G$
when $G$ is not a forest.  
\end{problem}

\section{Other small subgraphs}\label{Ssmall}
We have in this paper studied the evolution of 
the number of visible edges as a given graph is uncovered randomly;
this number equals the number of visible copies of $\sK_2$.
The methods in \refSs{Spf1}--\ref{Spf2}
above can be used to show similar results for the number of
visible copies of other small graphs.
This does not seem to involve any new ideas, but the calculations become
long and tedious, with more cases to treat, 
and since the paper already is long enough, we give only a brief
sketch for one instance, the number of $\sK_3$ (triangles).

As in \refS{Spf1}, we let $G$ be a given non-random graph on $[n]$
and consider the \ctime{} version of the uncovering process.
Let $T(t)$ be the number of visible triangles in $G$ at time $t\in\oi$.
Similarly to \eqref{b1} we have, 
using \eqref{I}--\eqref{II} and symmetry, and
summing over ordered triples of distinct indices  $i,j,k\in[n]$ 
satisfying the indicated conditions,
\begin{align}\label{t1}
6T(t)&=\sum_{ij,ik,jk\in E} I_i(t)I_j(t)I_k(t)
=\sum_{ij,ik,jk\in E} \bigpar{\II_i(t)+t}\bigpar{\II_j(t)+t}\bigpar{\II_k(t)+t}
\notag\\&
=\sum_{ij,ik,jk\in E} \II_i(t)\II_j(t)\II_k(t)
+3t \sum_{ij,ik,jk\in E}\II_i(t)\II_j(t)
+3t^2 \sum_{ij,ik,jk\in E}\II_i(t)
+6t^3 T(1)
\notag\\&
=:
\TT_1(t)+3t \TT_2(t)+3t^2\TT_3+6t^3T(1),
\end{align}
where $T(1)$ is the (non-random) number of triangles in $G$.
We define
\begin{align}\label{t2a}
  \TTT_1(t)&:=(1-t)^{-3}\TT_1(t)
=\sum_{i,j,k:ij,ik,jk\in E} \III_i(t)\III_j(t)\III_k(t),
\\\label{t2b}
  \TTT_2(t)&:=(1-t)^{-2}\TT_2(t)
=\sum_{i,j:ij\in E} \gd_{ij}\III_i(t)\III_j(t)
\\  \label{t2c}
\TTT_3(t)&:=(1-t)^{-1}\TT_3(t)
=\sumin 2\eps_{i}\III_i(t),
\end{align}
where $\gd_{ij}$ is the number of common neighbours of $i$ and $j$, and 
$\eps_i$ is the number of triangles in $G$ that contain $i$.

Similarly to \refS{Spf1}, $\TTT_\ell$ ($\ell=1,2,3$) are martingales on $\oio$,
and \eqref{t1} together with \eqref{t2a}--\eqref{t2c}
yields a decomposition of $T(t)$ into them.
The quadratic (co)variations and their expectations
are found as in \refS{Spf1}; we have for example
\begin{align}
  [\TTT_1,\TTT_1]_t=\sumin \indic{T_i\le t}(1-T_i)\qww
\biggpar{3\sum_{j,k:ij,ik,jk\in E}\III_j(t)\III_k(t)}^2
\end{align}
and as a consequence
\begin{align}
\E  [\TTT_1,\TTT_1]_t&
%=\sumin \indic{T_i\le t}(1-T_i)\qww
%\lrpar{3\sum_{ij,ik,jk\in E}\III_j(t)\III_k(t)}^2
= \sumin\int_0^t (1-s)^{-2}\cdot 36 \eps_i\E
  \bigpar{\III_1(s)\III_2(s)}^2\dd s
\notag\\&
= 36\sumin\eps_i\int_0^t \frac{s^2}{(1-s)^{4}}\dd s
\notag\\&
= 12\sumin \eps_i\cdot\frac{t^3}{(1-t)^3}
= 36T(1)\frac{t^3}{(1-t)^3}
.\end{align}
Furthermore, it seems clear that
it is possible to 
estimate the variances of the quadratic (co)variations by 
arguments similar to the ones in the proof of \refL{LV}.
Then, for a sequence of graphs $G\nn$ with
suitable hypotheses on vertex degrees
and on other small structures in $G\nn$ (in particular the number of triangles
and the numbers of pairs of triangles sharing one or two vertices), 
\refP{P:JS} and \refL{Loi} would apply and show joint convergence,
 after suitable normalization,
of $\TTT_\ell(t)$ in $D\oio$ and of $\TT_\ell(t)$ i $D\oi$,
which by the decomposition \eqref{t1} would yield convergence of $T(t)-t^3T(1)$
to a Gaussian process.
Furthermore, corresponding results for $\iT_\fnt$ would follow from
\refT{TN}.

We have, however, not checked the details, nor found a precise set of 
conditions; we leave this to vigorous readers to explore further.

\appendix

\section{Derandomizing time}\label{AA}

We consider in this appendix the problem of recovering results
for a \dtime{} stochastic process such as $\iL_k$ from results for the
corresponding \ctime{} process  $L(t)$.
We adapt the method from \eg{} \cite{SJ79,SJ94} to the present situation.

We state the result  generally. We assume that
(for each $n\ge1$)
we have a given set of $n$ elements, which we may identify with $[n]$,
and that we draw its elements one by one in random order
(i.e., uniformly at random and without replacement);
we assume also that we have a \dtime{} stochastic process
$(\iX_k)_{k=0}^n=(\iX_k\nn)_{k=0}^n$,
where  $\iX_k$ is the value of some variable
when we have drawn $k$ objects.
In the corresponding \ctime{} model, 
we give, as in \refS{Snot}, each element $i\in[n]$ a random variable
$T_i\in\Uoi$
representing the time when $i$ is drawn; we assume that $T_1,\dots,T_n$ are
independent. Then $X(t)=X\nn(t)$ is the value of our variable at time $t\in\oi$.
(The variables $\iX_k$ and $X(t)$ may depend on other underlying random
variables too; in that case, these variables are assumed to be independent
of the order the elements are drawn and of $(T_i)_1^n$.)

We define $N(t)=N\nn(t)$ and $\tau_k=\tau\nn_k$ by \eqref{nt} and
\eqref{tauk}, and note that,
as in \eqref{a5}, we may for each $n$ couple the \dtime{} and \ctime{}
processes such that
\begin{align}\label{aa1}
  X(t)=\iX_{N(t)} 
\qquad\text{and conversely}\qquad
\iX_k=X\xpar{\tau_k},
\end{align}
for all $t\in\oi$ and $k=0,\dots,n$, respectively.
Note that the process $(N(t))_{t\in\oi}$, which describes the collection of times
$\set{T_i}_1^n$, 
is stochastically independent of the order of these times, and thus of the
process $(\iX_k)_k$.
Moreover, $n\qw N(t)$ is the empirical distribution of $\set{T_i}_1^n$, and
thus,
as noted in \refR{RBB}, 
it is well known, see \eg{} \cite[Theorems 16.4 and 13.1]{Billingsley},
that, as \ntoo, 
\begin{align}\label{bra}
  n\qqw\bigpar{N\nn(t)-nt}\dto \BR(t)
\qquad\text{in $D\oi$},
\end{align}
where $\BR$  is a Brownian bridge, 
see \refS{Snot1} and in particular \eqref{brb}.
(In fact, we have proved \eqref{bra} as part of \refLs{LA}, \ref{LB}, and
\ref{LC}.)

We say that a function $f(t)$ is continuously differentiable on $\oi$ if it
is continuously differentiable in $(0,1)$ and $f'(t)$ extends continuously
to $\oi$.

\begin{theorem}\label{TN}
Suppose that, for each $n$,
$(\iX\nn_k)_{k=0}^n$ and $(X\nn(t))_{t\in\oi}$ are  stochastic processes
as above; in particular we assume that \eqref{aa1} holds.
Suppose also that $(a_n)\xoo$ and $(b_n)\xoo$ are sequences of positive
numbers, that $f(t)$ is a continuously differentiable function on $\oi$, 
and that 
$\bigpar{Z(t),W(t)}_{t\in\oi}$ is a continuous $2$-dimensional 
Gaussian
process on $\oi$ such that, 
as \ntoo,
in $D\oi$,
\begin{align}
  \label{tn1}
\bigpar{a_n (X\nn(t)-b_n f(t)),n\qqw(N\nn(t)-nt)} \dto\bigpar{Z(t),W(t)}.
\end{align}
%Suppose further that $a_nb_n/\sqrt n\to c\in[0,\infty]$.
Suppose further that $n\qqw a_nb_n\to c\in[0,\infty)$.
Then,
\begin{align}
  \label{tn2}
a_n \bigpar{\iX\nn_{\floor{nt}}-b_n f(t)}
\dto \iZ(t):=Z(t) - cf'(t)W(t)
\qquad\text{in $D\oi$}.
\end{align}
Moreover, 
$\iZ(t)$ is also a continuous Gaussian process on $\oi$, 
it has mean $\E\iZ(t)=\E Z(t)$, and
covariance function
\begin{align}\label{tnc}
  \Cov(\iZ(s),\iZ(t)) = \Cov(Z(s),Z(t)) - c^2s(1-t)f'(s)f'(t),
\qquad 0\le s\le t\le 1.
\end{align}
%%
%%If\/ $c=\infty$, we instead have
%%\begin{align}
%%  \label{tn3}
%%\frac{n\qq}{b_n} \bigpar{\iX\nn_{\floor{nt}}- b_n f(t)}
%%\dto  - f'(t)W_t
%%\qquad\text{in $D\oi$}
%%.\end{align}
%%\REM{Impossible, unless $f'(t)\equiv0$,
%%since otherwise $W_t$ independent of itself!?}
\end{theorem}

\begin{remark}
  By \eqref{bra}, \eqref{tn1} implies that 
$(W(t))_{t\in\oi}\eqd(\BR(t))_{t\in\oi}$, 
so $W(t)$ is
  just a Brownian bridge.
Nevertheless, we keep the notation $W(t)$ since in general $Z(t)$ and $W(t)$
are dependent, and their joint distribution is important in \eqref{tn2}.
We may also note that \eqref{tn1} is equivalent to
the limits
\eqref{bra} and
\begin{align}\label{sy1}
a_n \bigpar{X\nn(t)-b_n f(t)} \dto Z(t)
\end{align}
holding jointly, for some particular coupling of $Z(t)$ and
$\BR(t)$ (which we then denote by $W(t)$).
\end{remark}

\begin{proof}
By replacing $\iX\nn_k$, $X\nn(t)$, and $b_n$ by $a_n\iX\nn_k$,
$a_nX\nn(t)$, and $a_nb_n$, respectively, 
we may for convenience assume that $a_n=1$ for all $n$.

By the Skorohod coupling theorem \cite[Theorem~4.30]{Kallenberg},
we may assume that the limit in \eqref{tn1} holds a.s.
Since convergence in $D\oi$ to a continuous limit is equivalent to uniform
convergence, this means that \as, as \ntoo,
\begin{align}
 X\nn(t)&=b_n f(t)+ Z(t) + o(1),\label{aa3}
\\
N\nn(t)&=nt+n\qq W(t)+o\bigpar{n\qq},\label{aa4}
\end{align}
uniformly for $t\in\oi$.
In particular, substituting $t=\tau\nn_k$ in \eqref{aa4}, we obtain, \as, 
\begin{align}\label{aa5a}
  k=N\nn\xpar{\tau\nn_k}
=n\tau\nn_k+n\qq W\xpar{\tau\nn_k}+o\bigpar{n\qq},
\end{align}
and thus
\begin{align}\label{aa5}
\tau\nn_k=\frac{k}{n}
-n\qqw W\xpar{\tau\nn_k}+o\bigpar{n\qqw},
\end{align}
uniformly for $0\le k\le n<\infty$.
Since $W(t)$ is a continuous function of $t$, it is bounded
(with a random bound), and thus
\eqref{aa5} implies in particular that \as
\begin{align}\label{aa6}
  \tau\nn_k=\frac{k}{n}+o(1)\qquad\text{uniformly in $k=0,\dots,n$}
.\end{align}
Furthermore, $W(t)$ is uniformly continuous on the compact interval $\oi$,
and thus \eqref{aa6} implies $W\xpar{\tau\nn_k}-W\xpar{k/n}=o(1)$,
uniformly in $k$.
Consequently, using \eqref{aa5} again, \as
\begin{align}\label{aa7}
\tau\nn_k
=\frac{k}{n}-n\qqw W(k/n)+o\bigpar{n\qqw},
\end{align}
and similarly, uniformly for  $t\in\oi$,
\begin{align}\label{aa8}
\tau\nn_{\floor{nt}}
=t-n\qqw W_{t}+o\bigpar{n\qqw}
=t+o(1).
\end{align}
We substitute this in \eqref{aa3} and obtain \as, 
recalling \eqref{aa1}
and the fact that $Z(t)$ is uniformly continuous, 
and using a Taylor expansion of $f$,
\begin{align}\label{aa9}
  \iX\nn_{\floor{nt}}&
=X\nn\xpar{\tau\nn_{\floor{nt}}}
=b_nf\bigpar{\tau\nn_{\floor{nt}}}+Z\xpar{\tau\nn_{\floor{nt}}}+o(1)
\notag\\&
=b_nf\bigpar{t-n\qqw W(t)+o\bigpar{n\qqw}}+Z(t)+o(1)
\notag\\&
=b_nf(t)- b_n f'(t)n\qqw W(t)+o\bigpar{b_nn\qqw}+Z(t)+o(1),
\end{align}
uniformly for $t\in\oi$.
The result \eqref{tn2} follows, since $b_nn\qqw=a_nb_nn\qqw\to c$,
and in particular $b_nn\qqw=O(1)$.

%We may note that the proof shows that the limit \eqref{tn2} or \eqref{tn3}
%holds jointly with \eqref{tn1}.

Finally,
we recall that, as noted above,  $(\iX\nn_k)_k$ and $(N\nn(t))_t$ are independent
for each $n$.
Since the proof shows that the limits \eqref{tn1} and \eqref{tn2} 
hold jointly, we conclude that the limits $(\iZ(t))$ and $(W(t))$
are independent. Since
$Z(t)=\iZ(t)+cf'(t)W(t)$
by  \eqref{tn2}, it follows that
%when the variances are finite,
\begin{align}\label{tn1c}
  \Cov(Z(s),Z(t)) = \Cov(\iZ(s),\iZ(t)) + c^2f'(s)f'(t)s(1-t),
\qquad 0\le s\le t\le 1,
\end{align}
and thus \eqref{tnc} holds.
\end{proof}

\begin{remark}
As just noted, the proof shows that the limit \eqref{tn2} %or \eqref{tn3}
holds jointly with \eqref{tn1}, and furthermore that $(\iZ(t))$ is
independent of $(W(t))$. In particular, for all $s,t\in\oi$, we have
$\Cov(\iZ(s),W(t))=0$, and thus by \eqref{tn2} and \eqref{brb} necessarily
\begin{align}
  \Cov(Z(s),W(t))=cf'(s)\Cov(W(s),W(t))
=
  \begin{cases}
    cf'(s)s(1-t), & 0\le s\le t\le 1,
\\ 
cf'(s)(1-s)t,
& 0\le t\le s\le 1.
  \end{cases}
.\end{align}
\end{remark}

We note also that a converse to \refT{TN} holds.

\begin{theorem}\label{TN2}
Suppose that (as in \refT{TN})
$(\iX\nn_k)_0^n$ and $(X\nn(t))_{t\in\oi}$ are
stochastic processes as above, that 
$(a_n)\xoo$ and $(b_n)\xoo$ are positive numbers, and that $f(t)$ is a 
continuously differentiable function on $\oi$.
Suppose further that $(\iZ(t))_{t\in\oi}$ is a continuous Gaussian process
such that
\begin{align}
  \label{tn2a}
a_n \bigpar{\iX\nn_{\floor{nt}}-b_n f(t)}
\dto \iZ(t) 
\qquad\text{in $D\oi$}.
\end{align}
Suppose also that $a_nb_n/\sqrt n\to c\in\ooo$.
Then
\begin{align}
  \label{tn2b}
a_n \bigpar{X\nn\xpar{t}-b_n f(t)}
\dto Z(t):=\iZ(t) +cf'(t)\BR(t)
\qquad\text{in $D\oi$},
\end{align}
where $(\BR(t))_{t\in\oi}$ is a Brownian bridge independent of $(\iZ(t))_{t\in\oi}$.

Moreover, %if\/ $\E Z(t)^2<\infty$ for every $t$, then 
$Z(t)$ is also a continuous Gaussian process on $\oi$, 
it has mean $\E Z(t)=\E \iZ(t)$, and
%$Z(t)$ has 
covariance function
\begin{align}\label{tn2c}
  \Cov(Z(s),Z(t)) = \Cov(\iZ(s),\iZ(t)) + c^2s(1-t)f'(s)f'(t),
\qquad 0\le s\le t\le 1.
\end{align}
\end{theorem}

\begin{proof}
We argue as in the proof of \refT{TN}, and we may again assume $a_n=1$.
First, as noted above,  $(\iX\nn_k)_k$ and $(N\nn(t))_t$ are independent
for each $n$; hence,
the limits \eqref{tn2a} and \eqref{bra} hold jointly, with
  independent limits $(\iZ(t))_t$ and $(\BR(t))_t$.
Consequently, we may by the
Skorohod coupling theorem %\cite[Theorem~4.30]{Kallenberg},
assume that both \eqref{tn2a} and \eqref{bra} hold a.s.,
and thus
\begin{align}
\iX\nn_\fnt&=b_n f(t)+ \iZ(t) + o(1),\label{ab3}
\\
N\nn(t)&=nt+n\qq \BR(t)+o\bigpar{n\qq},\label{ab4}
\end{align}
uniformly for $t\in\oi$.
Consequently, a.s.,
\begin{align}
 n\qw N\nn(t)&=t+n\qqw \BR(t)+o\bigpar{n\qqw},\label{abc4}
\end{align}
and,
using \eqref{aa1}, the uniform continuity of $\iZ(t)$, and the
continuous differentiability of $f$,
\begin{align}
  X\nn(t)&
=\iX\nn_{N\nn(t)}
=b_n f\bigpar{n\qw N\nn(t)} + \iZ\bigpar{n\qw N\nn(t)}+o(1)
\notag\\&
=b_n f(t) + b_n  f'(t)n\qqw \BR(t)+o\bigpar{b_nn\qqw} + \iZ\bigpar{t}+o(1)
\end{align}
uniformly for $t\in\oi$,
which yields \eqref{tn2b}.

The covariance formula \eqref{tn2c} follows from \eqref{brb} and the
independence of $(\iZ(t))_t$ and $(\BR(t))_t$.
\end{proof}

\begin{remark}
\refTs{TN} and \ref{TN2} extend to vector-valued $\iX_k$ and $X(t)$, 
\emph{mutatis mutandis};
we may assume that all $\iX_k$ and $X(t)$
take their values in $\bbR^q$ for some fixed
$q\ge1$ (not depending on $k$ or $n$), and that then also $f(t)$ and $Z(t)$ or
$\iZ(t)$ take
their values in $\bbR^q$. 
%(Thus $(Z(t),W(t))$ in \refT{TN} becomes $(q+1)$-dimensional.)
\end{remark}

\begin{remark}\label{RGauss}
We have assumed in \refT{TN}
that $(Z(t),W(t))$ is a Gaussian process, since this is  the case  we use.
The theorem is valid (with the same proof) for any continuous stochastic process
$(Z(t),W(t))$, except that (of course)
$\iZ(t)$ then is not necessarily Gaussian, and that \eqref{tnc} requires
that the processes have finite variances.

Similarly, \refT{TN2} holds for any continuous stochastic process
$\iZ(t)$, with corresponding modifications.
\end{remark}

\section{On degree distributions in some random trees}\label{Atrees}
We apply our main results to several classes of random trees in examples in
\refS{Sex}. In order to do so, we have to verify the condition \eqref{ta2},
which says that the 
average of the squared degrees of the vertices converges to $\chix$ in
probability.
While this is closely related to known results on the degree distribution in
the random trees, we do not know any references stating precisely this result.
However, in this appendix we show that for the random trees considered in
\refS{Sex}, 
\eqref{ta2} easily follows from known results.
We also show that the condition \eqref{ta3} on the maximum degree holds for
these trees.

Throughout this section, we assume that $G\nn$ ($n\ge1$) is some sequence of
random trees with $V(G\nn)=[n]$.
Let $D_n$ be the degree of a %uniformly 
random vertex of $G\nn$.
Then \eqref{ta2} can be written
\begin{align}
  \label{tay}
\E\bigpar{D_n^2\mid G\nn} \pto\chix.
\end{align}

Our uncovering problem and the results for it in \refS{Smain} are stated for
unrooted graphs and trees, but 
they can of course be applied also to rooted trees by regarding them as
unrooted,  forgetting the choice of root.
Indeed, most of our examples of random trees in \refS{Sex} 
consider rooted trees.
For rooted trees, it is usually more convenient to consider outdegrees.
We assume (without loss of generality)
that the root is vertex 1, and we denote the outdegree of vertex
$i$ by $\hd_i$.
Thus 
\begin{align}\label{bb1}
  d_i=\hd_i+\indic{i\neq 1}.
\end{align}
Similarly,
we let $\hD_n$ denote the outdegree of a %(uniformly) 
random vertex of $G\nn$.
We then have the following simple reformulations of \eqref{ta2} and \eqref{tay}.
\begin{lemma}
  Suppose that $G\nn$ is a sequence of rooted trees with $V(G\nn)=[n]$, and
  suppose that \eqref{ta3} holds. Then \eqref{ta2} is equivalent to
  \begin{align}
    \label{htax}
\hchi\nn:=
\frac{1}{n}\sumin \bigpar{\hd\nn_i}^2\pto\hchix:=\chix-3,
  \end{align}
and thus also to
\begin{align}
  \label{htay}
\E\bigpar{\hD_n^2\mid G\nn} \pto\hchix.
\end{align}
\end{lemma}
\begin{proof}
By \eqref{bb1}, we have, using $\sumin\hd_i\nn=n-1$ and \eqref{ta3},
\begin{align}
  \sumin \bigpar{d\nn_i}^2&
=\sumin \Bigpar{\bigpar{\hd\nn_i}^2+2\hd\nn_i+1}
-2\hd\nn_1-1
\notag\\&
=\sumin \bigpar{\hd\nn_i}^2+2(n-1)+n
-2\hd\nn_1-1
\notag\\&
=\sumin \bigpar{\hd\nn_i}^2+3n+o(n).
\end{align}
The result follows by dividing by $n$.
\end{proof}

\subsection{Preliminaries}
We let $\cP(\bbN)$ be
the space of probability distributions on $\bbN=\set{0,1,\dots}$. 
We give $\cP(\bbN)$ the standard weak topology (for the space of
probability measures on any metric space), see \eg{} \cite{Billingsley}.
Thus, since $\bbN$ is discrete,
if $X_n$ ($n\ge1$) and $X$ are random variables with values in $\bbN$, then
convergence in $\cP(\bbN)$ of the
distributions 
\begin{align}
  \label{bb00}
\cL(X_n)\to\cL(X)
\end{align}
means
\begin{align}\label{bb01}
  \E f(X_n)\to\E f(X)
\qquad\text{for every bounded function $f:\bbN\to\bbR$}.
\end{align}
%for every bounded function $f:\bbN\to\bbR$.
It is well known \cite[Theorem 5.6.4]{Gut}
that, again since $\bbN$ is discrete,
this is equivalent both to 
convergence of point probabilities 
\begin{align}\label{bb02}
\P(X_n=d)\to\P(X=d)  
\qquad\text{for every $d\in\bbN$},
\end{align}
and 
also to convergence in total variation:
\begin{align}\label{bb03}
\sum_{d=0}^\infty\bigabs{\P(X_n=d)-\P(X=d)}\to0.  
\end{align}

Consider now a sequence $G\nn$ of random trees, as usual with 
$V(G\nn)=[n]$. %$|G\nn|=n$
Then the conditional distribution $\cL(D_n\mid G\nn)$ is a random
distribution on $\bbN$, \ie. a random element of $\cP(\bbN)$.
The equivalence of \eqref{bb00}--\eqref{bb03} above 
transfers to convergence in probability of random distributions in $\cP(\bbN)$,
and in particular, in our situation, we have the following
(note that  the \rhs{} of \eqref{bp00} is a constant element of
$\cP(\bbN)$):
\begin{lemma}\label{LBB0}
For any random variable $\zeta\in\bbN$, the following are equivalent:
\begin{align}  \label{bp00}
\cL\bigpar{D_n\mid G\nn}&\pto \cL(\zeta), 
\\\label{bp01}
  \E\bigpar{f(D_n)\mid G_n}&\pto\E f(\zeta)
\quad\text{for every bounded $f:\bbN\to\bbR$},
\\  \label{bp02}
\P\bigpar{D_n=d\mid G\nn}&\pto\P(\zeta=d)  
\quad\text{for every fixed $d\in\bbN$},
\\ \label{bp03}
\sum_{d=0}^\infty\bigabs{\P(D_n=d\mid G\nn)&-\P(\zeta=d)}\pto0.  
\end{align}
If\/ $G\nn$ are rooted trees, we also have the same equivalences with $\hD_n$
instead of $D_n$.
\end{lemma}
\begin{proof}
Using the fact that (in any metric space) a sequence converging in
probability has a subsequence converging a.s.,
  this follows easily from the 
 equivalence of \eqref{bb00}--\eqref{bb03};
we omit the details.
\end{proof}

We state a simple lemma in a general form; 
recall that $\gD\nn$ denotes the maximum degree in $G\nn$.
We are mainly interested in the case
$f(x)=x^2$;  note that in this case, 
\eqref{lbb3} and \eqref{lbb4} are equivalent to
\eqref{ta2} and \eqref{ta3}, with $\chix=\E f(\zeta)=\E \zeta^2$.

\begin{lemma}\label{LBB}
Let $\zeta$ be a random variable with values in $\bbN$
and  let $f:\bbN\to\ooo$ be any function such that $\E f(\zeta)<\infty$.
If, as \ntoo,
\begin{align}\label{lbb1}
  \cL(D_n\mid G\nn)\pto \cL(\zeta)
\end{align}
and
\begin{align}\label{lbb2}
\E f(D_n)\to \E f(\zeta),  
\end{align}
then
\begin{align}\label{lbb3}
\E(f(D_n)\mid G\nn)&\pto \E f(\zeta),
\\ \label{lbb4}
f(\gD\nn)/n&\pto0. %=\op(n).
\end{align}
If\/ $G\nn$ are rooted trees, we also have the same equivalences with $\hD_n$
instead of $D_n$
and $\hgD\nn:=\max_i\hd\nn_i\ge \gD\nn-1$ instead of $\gD\nn$.
\end{lemma}
\begin{proof}
We use truncations. For $M\ge0$, let $f_M(x):=f(x)\bmin M$.
Then each $f_M$ is bounded, and thus by 
\eqref{lbb1} and \refL{LBB0},
\begin{align}\label{lbc1}
\E(f_M(D_n)\mid G\nn)\pto \E f_M(\zeta).
\end{align}
Furthermore, taking the expectation in \eqref{lbc1}, we obtain by dominated
convergence 
(for convergence in probability, see \eg{} \cite[Theorem 5.5.4]{Gut}), 
again because $f_M$ is bounded,
\begin{align}\label{lbc2}
\E f_M(D_n)\to \E f_M(\zeta).
\end{align}

We have $f_M(D_n)\le f(D_n)$ and thus, using Markov's inequality,
for any $\eps>0$.
\begin{align}
  \P\bigsqpar{\E(f(D_n)\mid G\nn)&-\E(f_M(D_n)\mid G\nn)>\eps}
\notag\\&=
  \P\bigsqpar{\E(f(D_n)- f_M(D_n)\mid G\nn)>\eps}
\notag\\&
\le \eps\qw \E\bigsqpar{\E(f(D_n)- f_M(D_n)\mid G\nn)}
\notag\\&
=
 \eps\qw \E\bigsqpar{f(D_n)- f_M(D_n)}.
\end{align}
Hence, by \eqref{lbb2} and \eqref{lbc2},
\begin{align}\label{lbc4}
&\limsup_\ntoo  \P\bigsqpar{\E(f(D_n)\mid G\nn)-\E(f_M(D_n)\mid G\nn)>\eps}
\notag\\&\qquad
\le 
\eps\qw \limsup_\ntoo\E\bigsqpar{f(D_n)- f_M(D_n)}
= \eps\qw \bigpar{\E f(\zeta)- \E f_M(\zeta)}.
\end{align}
By monotone convergence, 
\begin{align}
  \label{lbc5}
\E f_M(\zeta)\to \E f(\zeta)
\qquad \text{as $M\to\infty$},
\end{align}
and thus
\begin{align}\label{lbc6}
  \lim_{M\to\infty}\limsup_\ntoo  
\P\bigsqpar{\E(f(D_n)\mid G\nn)&-\E(f_M(D_n)\mid G\nn)>\eps}=0.
\end{align}
We have, for any $\eps>0$,
\begin{align}\label{lbc7}
\P\bigsqpar{\bigabs{\E(f(D_n)\mid G\nn)&-\E f(\zeta)}>3\eps}
\notag\\&
\le
\P\bigsqpar{\bigabs{\E(f(D_n)\mid G\nn)-\E(f_M(D_n)\mid G\nn)}>\eps}
\notag\\&\qquad\qquad+
\P\bigsqpar{\bigabs{\E(f_M(D_n)\mid G\nn)-\E f_M(\zeta)}>\eps}
\notag\\&\qquad\qquad+
\P\bigsqpar{\bigabs{\E f_M(\zeta)-\E f(\zeta)}>\eps}
\end{align}
Taking first the limsup as $\ntoo$, and then letting $\Mtoo$, 
we see 
from \eqref{lbc6}, \eqref{lbc1}, and \eqref{lbc5}
that the \rhs{} tends to 0, 
which yields \eqref{lbb3}.
(Alternatively, since convergence in probability to a constant is the same
as convergence in distribution to the same constant, 
\cite[Theorem 4.2]{Billingsley} shows that
%(or \cite[Theorem 4.28]{Kallenberg}),
\eqref{lbc5} and \eqref{lbc6}
enable us to let $M\to\infty$ in \eqref{lbc1},
which yields \eqref{lbb3}.)

Moreover, $\P\bigpar{D_n=\gD\nn\mid G\nn} \ge 1/n$, and thus,
for any $M\ge0$,
\begin{align}
  f(\gD\nn) \le  f(\gD\nn)-f_M(\gD\nn)+M
\le  n \E \bigpar{f(D_n)-f_M(D_n)\mid G\nn}+M.
\end{align}
Hence, for any $\eps>0$ and $M\ge0$, we have for $n>2M/\eps$,
\begin{align}
\P\sqpar{f(\gD\nn)/n>\eps}
\le \P\bigsqpar{\E \bigpar{f(D_n)-f_M(D_n)\mid G\nn}>\eps/2},
\end{align}
 and consequently, %by \eqref{lbc4},
\begin{align}
\limsup_\ntoo
\P\sqpar{f(\gD\nn)/n>\eps}
\le 
\limsup_\ntoo\P\bigsqpar{\E \bigpar{f(D_n)-f_M(D_n)\mid G\nn}>\eps/2}.
%\le 2 \eps\qw \bigpar{\E f(\zeta)- \E f_M(\zeta)}.
\end{align}
This holds for every $M\ge0$, and thus \eqref{lbc6} shows that
\begin{align}
\lim_\ntoo \P\bigpar{f(\gD\nn)/n>\eps}=0
\end{align}
for every $\eps>0$, which is \eqref{lbb4}.

The proof for $\hD\nn$ is the same.
\end{proof}

\subsection{Conditioned Galton--Watson trees}\label{SSGW}

Let, as in \refE{EGW}, $G\nn$ be a \cGWt{} with $n$ vertices
defined with offspring distributed as a random variable  $\xi\in\bbN$ with
$\E\xi=1$ and 
$0<\Var\xi<\infty$.
Recall that this means that $G\nn$
is obtained by conditioning a \GWt{} $\cT$ on having exactly $n$ vertices,
where in $\cT$, every vertex has a number of children distributed as
independent copies of $\xi$;
see \eg{} \cite{SJ264}.
We denote the distribution of $\xi$ by $\cL(\xi)\in\cP(\bbN)$; this is known
as the offspring distribution. 

We regard the trees $G\nn$ and $\cT$ as rooted and ordered.

It is well known that the asymptotic degree distribution 
of $G\nn$ is the offspring distribution $\cL(\xi)$, \ie,
$\hD_n\dto\xi$ as \ntoo;
moreover, this holds also (in probability) conditioned on $G\nn$, \ie,
\begin{align}\label{gw1}
  \cL(\hD_n\mid G\nn)\pto \cL(\xi),
\end{align}
see  \cite[Theorem 7.11(ii)]{SJ264}.
(The notation is different there,
but \cite[Theorem 7.11(ii)]{SJ264} is equivalent to 
$\P(\hD_n=d\mid G\nn)\pto \P(\xi=d)$, which is equivalent to \eqref{gw1}
by \refL{LBB0}.)
Furthermore, let $f(x):=\binom x2 = x(x-1)/2$. Then,
\begin{align}\label{gw2}
  \E\bigpar{f(\hD_n)\mid G\nn}
= \frac{1}n\sumin \frac{\hd_i(\hd_i-1)}{2}.
\end{align}
Denote the sum in \eqref{gw2} by $\gU_n$, and note that if $\bt$ is the rooted
tree consisting of a root with two children,
then $\gU_n$ is the number of  subtrees of $G\nn$ that are isomorphic
to $\bt$, where we consider general subtrees, and consider each subtree as
rooted with the same parent-child relations as in $G\nn$. 
This number $\gU_n$
is treated (for general rooted and ordered trees $\bt$) in 
\cite[Theorem 1.1]{SJ355},
which shows, in particular,
that under our conditions $\E\xi=1$ and $\E\xi^2<\infty$, 
we have
\begin{align}
\frac{  \gU_n }n &\pto \E \frac{\xi(\xi-1)}{2} = \E f(\xi),
\intertext{and}
\E  \frac{\gU_n }n &\tend \E \frac{\xi(\xi-1)}{2}= \E f(\xi).
\end{align}
By \eqref{gw2}, these are precisely \eqref{lbb3} and \eqref{lbb2}, for
$\hD_n$,
with $\zeta=\xi$.
Since trivially $\E D_n=(n-1)/n\to 1=\E\xi$,
\eqref{lbb2} is trivial for $\hD_n$ and $f(x)=x$.
Hence, still taking $\zeta=\xi$,
it follows by linearity that \eqref{lbb2} holds also for $f(x)=x^2$, and
thus, recalling \eqref{gw1}, \refL{LBB} shows \eqref{lbb3} and \eqref{lbb4}
for $\hD_n$ and $f(x)=x^2$. 
By linearity again, we may also take $f(x)=(x+1)^2$ in
\eqref{lbb3}, which (using \eqref{lbb4})
easily is seen to be equivalent to \eqref{lbb3} for
$D_n$ with $f(x)=x^2$ and $\zeta=\xi+1$. 
Consequently, we have \eqref{ta2} and \eqref{ta3}, with
$\chix=\E(\xi+1)^2$. Since $\dx=2=\E(\xi+1)$, it follows that
\begin{align}
  \gamx:=\chix-\dx^2=\Var(\xi+1)=\Var\xi.
\end{align}

\begin{remark}\label{RGWr}
  It follows similarly from \cite{SJ355} 
(by taking $\bt$ as a star with a root of degree $k\le r$)
that for any integer $r\ge2$, 
if  $\E\xi^r<\infty$, then 
\begin{align}
\frac{1}{n}\sumin \hd_i^r&=  \E\bigpar{\hD_n^r\mid G\nn}\pto \E \xi^r,
\\%\intertext{and}
\frac{1}{n}\sumin d_i^r&=  \E\bigpar{D_n^r\mid G\nn}\pto \E (\xi+1)^r.
\end{align}
In other words, the moments of the (out)degree distribution converge to the
moments of $\xi$ or $\xi+1$, provided the latter moments are finite.
\end{remark}

\subsection{The random recursive tree}\label{SSRRT}
Let now $G\nn$ be a random recursive tree.
It is well known that
the asymptotic outdegree distribution is geometric $\Ge(1/2)$;
moreover, this holds conditioned on $G\nn$ in the sense \eqref{bp02},
see \eg{} 
\cite{MM1988}, 
\cite[Section 3.2]{Aldous-fringe}, 
\cite[Theorem 1]{SJ155},
\cite[Theorem 6.8]{Drmota},
and 
\cite[Example 6.1]{SJ306}.
Hence, \refL{LBB0} shows that \eqref{bp00}--\eqref{bp03} hold for $\hD_n$ with
$\zeta=\xi\sim\Ge(1/2)$; consequently they also hold for the total degree $D_n$
with $\zeta=\xi+1$.

We next verify \eqref{lbb2}, again first taking
$f(x):=\binom x2 = x(x-1)/2$. 
Since $G\nn$ is constructed with vertices added in increasing order,
%$\binom{\hd_i}2$ 
$f(\hd_i)$ is the number of pairs $(j,k)$ with $i<j<k\le n$ such that
$ij$ and $ik$ are edges.
Hence, since $\P(j\sim i)=1/(j-1)$ for $j>i$, and these events for different
$j$ are independent,
\begin{align}
  n\E f(\hD_n)& 
= \sumin \E f(\hd_i)
= \sum_{1\le i<j<k\le n}\P(j\sim i \text{ and } k \sim i)
= \sum_{1\le i<j<k\le n}\frac{1}{(j-1)(k-1)}
\notag\\&
= \sum_{2\le j<k\le n}\frac{1}{k-1}
= \sum_{3\le k\le n}\frac{k-2}{k-1}
= n+O(\log n).
\end{align}
Consequently, $\E f(\hD_n)\to1$, which verifies \eqref{lbb2} since 
$\E f(\xi)=\frac12(\E\xi^2-\E\xi)=\frac12(\Var(\xi)+1-1)=1$.

We have shown \eqref{lbb2} with $f=x(x-1)/2$, and, 
as in \refSS{SSGW},
it follows by linearity
that
\eqref{lbb2} holds also for $f(x)=x^2$, and 
furthermore that \eqref{lbb2}--\eqref{lbb4} hold for $D_n$ and $f(x)=x^2$.
Hence, \eqref{ta2} and \eqref{ta3} hold, with $\chix=\E(\xi+1)^2$.

\begin{remark}
  A similar calculation shows that \eqref{lbb2} holds for $\hD_n$ and
  $f(x):=\binom xr$ for any integer $r\ge2$; it follows that \eqref{lbb2}
  holds also for $f(x)=x^r$, and thus \refL{LBB} shows that 
all moments of the outdegree distribution (given $G\nn$)
 converge in probability to the corresponding moments of $\xi\sim\Ge(1/2)$.
Hence, all moments of the  degree distribution $D_n$
 converge in probability to the corresponding moments of $\xi+1$.
\end{remark}

\newcommand\AAP{\emph{Adv. Appl. Probab.} }
\newcommand\JAP{\emph{J. Appl. Probab.} }
\newcommand\JAMS{\emph{J. \AMS} }
\newcommand\MAMS{\emph{Memoirs \AMS} }
\newcommand\PAMS{\emph{Proc. \AMS} }
\newcommand\TAMS{\emph{Trans. \AMS} }
\newcommand\AnnMS{\emph{Ann. Math. Statist.} }
\newcommand\AnnPr{\emph{Ann. Probab.} }
\newcommand\CPC{\emph{Combin. Probab. Comput.} }
\newcommand\JMAA{\emph{J. Math. Anal. Appl.} }
\newcommand\RSA{\emph{Random Structures Algorithms} }
\newcommand\DMTCS{\jour{Discr. Math. Theor. Comput. Sci.} }

\newcommand\AMS{Amer. Math. Soc.}
\newcommand\Springer{Springer-Verlag}
\newcommand\Wiley{Wiley}

\newcommand\vol{\textbf}
\newcommand\jour{\emph}
\newcommand\book{\emph}
\newcommand\inbook{\emph}
\def\no#1#2,{\unskip#2, no. #1,} %(typeset after year) 
\newcommand\toappear{\unskip, to appear}

\newcommand\arxiv[1]{\texttt{arXiv}:#1}
\newcommand\arXiv{\arxiv}

\newcommand\xand{and }
\renewcommand\xand{\& }

\def\nobibitem#1\par{}

\end{document}